\documentclass[12pt]{amsart}
\usepackage{amscd,amsmath,amsthm,amssymb,amsfonts,enumerate,hyperref}

\def\NZQ{\mathbb}               
\def\NN{{\NZQ N}}

\def\opn#1#2{\def#1{\operatorname{#2}}} 
	\opn\chara{char} \opn\length{\ell} \opn\pd{pd} \opn\rk{rk}
	\opn\projdim{proj\,dim} \opn\injdim{inj\,dim} \opn\rank{rank}
	\opn\depth{depth} \opn\grade{grade} \opn\height{height}
	\opn\embdim{emb\,dim} \opn\codim{codim}
	\opn\Cl{Cl}
	
	\opn\Tr{Tr} \opn\bigrank{big\,rank}
	\opn\superheight{superheight}\opn\lcm{lcm}
	\opn\trdeg{tr\,deg}
	\opn\rdeg{rdeg}
	\opn\reg{reg} \opn\lreg{lreg} \opn\ini{in} \opn\lpd{lpd}
	\opn\size{size} \opn\sdepth{sdepth}
	\opn\link{link}\opn\fdepth{fdepth}\opn\lex{lex}
	\opn\tr{tr}
	\opn\type{type}
	\opn\gap{gap}
	\opn\arithdeg{arith-deg}
	\opn\revlex{revlex}
	\opn\div{div} \opn\Div{Div} \opn\cl{cl} \opn\Cl{Cl}
	\opn\Spec{Spec} \opn\Supp{Supp} \opn\supp{supp} \opn\Sing{Sing}
	\opn\Ass{Ass} \opn\Min{Min}\opn\Mon{Mon}
	\opn\Ann{Ann} \opn\Rad{Rad} \opn\Soc{Soc}
	\opn\Im{Im} \opn\Ker{Ker} \opn\Coker{Coker} \opn\Am{Am}
	\opn\Hom{Hom} \opn\Tor{Tor} \opn\Ext{Ext} \opn\End{End}
	\opn\Aut{Aut} \opn\id{id}
	
	\opn\nat{nat}
	\opn\pff{pf}
	\opn\Pf{Pf} \opn\GL{GL} \opn\SL{SL} \opn\mod{mod} \opn\ord{ord}
	\opn\Gin{Gin} \opn\Hilb{Hilb}\opn\sort{sort}
	\opn\PF{PF}\opn\Ap{Ap}
	\opn\mult{mult}
	\opn\bight{bight}
	\opn\div{div}
	\opn\Div{Div}
	\opn\aff{aff}
	\opn\relint{relint} \opn\st{st}
	\opn\lk{lk} \opn\cn{cn} \opn\core{core} \opn\vol{vol}  \opn\inp{inp}
	\opn\nilpot{nilpot}
	\opn\link{link} \opn\star{star}\opn\lex{lex}\opn\set{set}
	\opn\width{wd}
	\opn\Fr{F}
	\opn\QF{QF}
	\opn\G{G}
	\opn\type{type}\opn\res{res}
	\opn\conv{conv}
	\opn\Int{Int}
	\opn\Deg{Deg}
	\opn\Sym{Sym}
	\opn\Con{Con}
	\opn\gr{gr}
	
	\def\pot#1#2{#1[\kern-0.28ex[#2]\kern-0.28ex]}

	\opn\dirlim{\underrightarrow{\lim}}
	\opn\inivlim{\underleftarrow{\lim}}
	\def\Implies{\ifmmode\Longrightarrow \else
		\unskip${}\Longrightarrow{}$\ignorespaces\fi}
	\def\implies{\ifmmode\Rightarrow \else
		\unskip${}\Rightarrow{}$\ignorespaces\fi}
	\def\iff{\ifmmode\Longleftrightarrow \else
		\unskip${}\Longleftrightarrow{}$\ignorespaces\fi}

	\let\:=\colon
	\newtheorem{Theorem}{Theorem}[section]
	\newtheorem{Lemma}[Theorem]{Lemma}
	
	\newtheorem{Proposition}[Theorem]{Proposition}
	\theoremstyle{definition}
	\newtheorem{Remark}[Theorem]{Remark}
	
	\newtheorem{Example}[Theorem]{Example}
	
    \newtheorem{Question}[Theorem]{Question}

    \makeatletter
\@namedef{subjclassname@2020}{%
  \textup{2020} Mathematics Subject Classification}
\makeatother

\begin{document}

\title[Bounded powers of edge ideals]{Bounded powers of edge ideals: Gorenstein toric rings}

\author[T.~Hibi]{Takayuki Hibi}
\author[S.~A.~ Seyed Fakhari]{Seyed Amin Seyed Fakhari}

\address{(Takayuki Hibi) Department of Pure and Applied Mathematics, Graduate School of Information Science and Technology, Osaka University, Suita, Osaka 565--0871, Japan}
\email{hibi@math.sci.osaka-u.ac.jp}
\address{(Seyed Amin Seyed Fakhari) Departamento de Matem\'aticas, Universidad de los Andes, Bogot\'a, Colombia}
\email{s.seyedfakhari@uniandes.edu.co}



\begin{abstract}
Let $S=K[x_1, \ldots,x_n]$ denote the polynomial ring in $n$ variables over a field $K$ and $I \subset S$ a monomial ideal.  Given a vector $\mathfrak{c}\in\NN^n$, the ideal  $I_{\mathfrak{c}}$ is the ideal generated by those monomials belonging to $I$ whose exponent vectors are componentwise bounded above by $\mathfrak{c}$.  Let $\delta_{\mathfrak{c}}(I)$ be the largest integer $q$ for which $(I^q)_{\mathfrak{c}}\neq 0$.  For a finite graph $G$, its edge ideal is denoted by $I(G)$. Let $\mathcal{B}(\mathfrak{c},G)$ be the toric ring which is generated by the monomials belonging to the minimal system of monomial generators of $(I(G)^{\delta_{\mathfrak{c}}(I)})_{\mathfrak{c}}$. In a previous work, the authors proved that $(I(G)^{\delta_{\mathfrak{c}}(I)})_{\mathfrak{c}}$ is a polymatroidal ideal.  It follows that $\mathcal{B}(\mathfrak{c},G)$ is a normal Cohen--Macaulay domain.  In this paper, we study the Gorenstein property of $\mathcal{B}(\mathfrak{c},G)$.
\end{abstract}

\maketitle

\section*{Introduction}
Let $S=K[x_1, \ldots,x_n]$ denote the polynomial ring in $n$ variables over a field $K$ and $I \subset S$ a monomial ideal. Also, let $\NN$ denote the set of positive integers.  Given a vector $\mathfrak{c}=(c_1, \ldots, c_n) \in\NN^n$, the ideal $I_{\mathfrak{c}} \subset S$ is the ideal generated by those monomials $x_1^{a_1} \cdots x_n^{a_n}$ belonging to $I$ with $a_i \leq c_i$, for each $i=1, \ldots, n$.  Let $\delta_{\mathfrak{c}}(I)$ be the largest integer $q$ for which $(I^q)_{\mathfrak{c}}\neq 0$.  

Let $G$ be a finite graph with no loop, no multiple edge and no isolated vertex on the vertex set $V(G)=\{x_1, \ldots, x_n\}$ and $E(G)$ the set of edges of $G$.  Recall that the edge ideal of $G$ is the monomial ideal $I(G) \subset S$ generated by those $x_ix_j$ with $\{x_i, x_j\} \in E(G)$.  Let $\{w_1, \ldots, w_s\}$ denote the minimal set of monomial generators of $(I(G)^{\delta_{\mathfrak{c}}(I)})_{\mathfrak{c}}$ and $\mathcal{B}(\mathfrak{c},G)$ the toric ring $K[w_1, \ldots, w_s]\subset S$.  In \cite{HSF}, it is proved that $(I(G)^{\delta_{\mathfrak{c}}(I)})_{\mathfrak{c}}$ is a polymatroidal ideal. It then follows from \cite[Theorem 12.5.1]{HHgtm260} that $\mathcal{B}(\mathfrak{c},G)$ is a normal Cohen--Macaulay domain.  Naturally, one can ask when $\mathcal{B}(\mathfrak{c},G)$ is Gorenstein.  More precisely, 

\begin{Question}
\label{question}
Given a finite graph $G$ on $V(G)=\{x_1, \ldots, x_n\}$, find all possible $\mathfrak{c} \in \NN^n$ for which $\mathcal{B}(\mathfrak{c},G)$ is Gorenstein.
\end{Question}

However, one cannot expect a complete answer to Question \ref{question}.  For example, when $G$ is the star graph on $V(G)=\{x_1, \ldots, x_n, x_{n+1}\}$ with the edges $\{x_i, x_{n+1}\}$, $1 \leq i \leq n$, the answer to Question \ref{question} is exactly the classification of Gorenstein algebras of Veronese type (Example \ref{Ex}).  Its classification achieved in \cite{DH} by using the techniques on convex polytopes (\cite[p.~251]{HHgtm260}) is rather complicated. 

In the present paper, after summarizing notations and terminologies of graph theory in Section $1$, in Section $2$, it is shown that (i) for every finite graph $G$ on $V(G)=\{x_1, \ldots, x_n\}$ there exists a vector $\mathfrak{c} \in \NN^n$ for which $\mathcal{B}(\mathfrak{c},G)$ is Gorenstein (Theorem \ref{Gc}) and (ii) for every vector $\mathfrak{c} \in \NN^n$ there exists a finite graph $G$ on $V(G)=\{x_1, \ldots, x_n\}$ for which $\mathcal{B}(\mathfrak{c},G)$ is Gorenstein (Theorem \ref{cG}).  The highlight of the present paper is Section $3$, where it is proved that a finite graph $G$ on $V(G)=\{x_1, \ldots, x_n\}$ possesses the distinguished property that $\mathcal{B}(\mathfrak{c},G)$ is Gorenstein for all $\mathfrak{c} \in \NN^n$ if and only if there is an integer $t \geq 3$ such that each connected component of $G$ is either $K_2$ or $K_t$ (Theorem \ref{highlight}), where $K_t$ is the complete graph on $t$ vertices.  Finally, in Section $4$, we discuss Question \ref{question} for special classes of finite graphs. By virtue of the criterion \cite[Theorem 2.4]{DH} of Gorenstein algebras of Veronese type, we classify $\mathfrak{c}=(c_1, \ldots, c_{n})\in\NN^n$ for which $\mathcal{B}(\mathfrak{c},K_{n_1,\ldots, n_m}-M)$ is Gorenstein (Theorem \ref{completeAVT}), where $K_{n_1,\ldots, n_m}$ is a complete multipartite graph and $M$ is a (possibly empty) matching of it. Furthermore, by using the classification \cite[Remark 2.8]{oh} of Gorenstein edge rings of complete multipartite graphs, we classify trees $T$ on $n$ vertices satisfying ${\rm match}(T)=(n-2)/2$ for which $\mathcal{B}((1,1,\ldots,1),T)$ is Gorenstein (Theorem \ref{treegor}).

\section{Preliminaries}
We summarize notations and terminologies on finite graphs. Let $G$ be a finite graph with no loop, no multiple edge and no isolated vertex on the vertex set $V(G)=\{x_1, \ldots, x_n\}$ and $E(G)$ the set of edges of $G$.   

\begin{itemize}
 \item 
We say that $x_i \in V(G)$ is {\em adjacent} to $x_j \in V(G)$ in $G$ if $\{x_i,x_j\} \in E(G)$.  In addition, $x_j$ is called a {\em neighbor} of $x_i$. Let $N_G(x_i)$ denote the set of vertices of $G$ to which $x_i$ is adjacent. The cardianlity of $N_G(x_i)$ is the {\em degree} of $x_i$, denoted by ${\rm deg}_G(x_i)$. A {\em leaf} of $G$ is a vertex of degree one. Furthermore, if $A \subset V(G)$, then we set $N_G(A) := \cup_{x_i \in A} N_G(x_i)$.
\item 
We say that $e \in E(G)$ is {\em incident} to $x \in V(G)$ if $x \in e$.
 
    \item 
The {\em complete graph} $K_n$ is the finite graph on $V(K_n)=\{x_1, \ldots, x_n\}$ with $E(K_n) = \{ \{x_i,x_j\} : 1 \leq i < j \leq n \}$.  The {\em complete bipartite graph} $K_{n,m}$ is the finite graph on $V(K_{n,m})=\{x_1, \ldots, x_n\}\sqcup\{y_1, \ldots, y_m\}$ with $E(K_{n,m}) = \{ \{x_i, y_j\} : 1 \leq i \leq n, 1 \leq j \leq m \}$. The graph $K_{1,n}$ is called a {\em star} graph. In this case, the vertex of degree $n$ is the {\em center} of the graph.
 \item 
A {\em forest} is a finite graph with no cycle.  A {\em tree} is a connected forest.
 \item 
 A subset $C \subset V(G)$ is called {\em independent} if $\{x_i, x_j\} \not\in E(G)$ for all $x_i, x_j \in C$ with $x_i \neq x_j$. 
    \item 
A {\em matching} of $G$ is a subset $M \subset E(G)$ for which $e \cap e' = \emptyset$ for $e, e' \in M$ with $e \neq e'$.  We say that a matching $M$ of $G$ {\em covers} $x \in V(G)$ if there is $e \in M$ with $x \in e$.  The {\em matching number} of $G$ is the biggest possible cardinality of matchings of $G$.  Let ${\rm match}(G)$ denote the matching number of $G$.  A {\em maximal matching} of $G$ is a matching $M$ of $G$ for which there is no matching $M'$ of $G$ with $M \subsetneq M'$.  A {\em maximum matching} of $G$ is a matching $M$ of $G$ with $|M| = {\rm match}(G)$.  Every maximum matching is a maximal matching.  The {\em perfect matching} of $G$ is a matching $M$ of $G$ with $\cup_{e \in M}e = V(G)$.     
    \item 
If $M$ is a matching of $G$, then we define $G-M$ to be the finite graph obtained from $G$ by removing all edges belonging to $M$.  
\item 
If $U \subset V(G)$, then $G -  U$ is the finite graph on $V(G)\setminus U$ with $E(G - U) = \{e \in E(G) : e \cap U = \emptyset\}$.  In other words, $G - U$ is the {\em induced subgraph} $G_{V(G)\setminus U}$ of $G$ on $V(G)\setminus U$.   
    \item 
In the polynomial ring $S = K[x_1, \ldots, x_n]$, unless there is a misunderstanding, for an edge $e = \{x_i, x_j\}$, we employ the notation $e$ instead of the monomial $x_ix_j\in S$.  For example, if $e_1 = \{x_1, x_2\}$ and $e_2 = \{x_2, x_5\}$, then $e_1^2e_2 = x_1^2x_2^3x_5$. 
\end{itemize}

\section{Exsistence}
First of all, we show that (i) for every finite graph $G$ on $V(G)=\{x_1, \ldots, x_n\}$, there is a vector $\mathfrak{c} \in \NN^n$ for which $\mathcal{B}(\mathfrak{c},G)$ is Gorenstein and (ii) for every vector $\mathfrak{c} \in \NN^n$ there is a finite graph $G$ on $V(G)=\{x_1, \ldots, x_n\}$ for which $\mathcal{B}(\mathfrak{c},G)$ is Gorenstein. 

\begin{Theorem}
\label{Gc}
Given a finite graph $G$ on $V(G)=\{x_1, \ldots, x_n\}$, there is a vector $\mathfrak{c}\in \NN^n$ for which $\mathcal{B}(\mathfrak{c},G)$ is Gorenstein.
\end{Theorem}

\begin{proof}
Set $u:=\prod_{\{x_i,x_j\}\in E(G)}(x_ix_j)$ and $\mathfrak{c}$ the exponent vector of $u$.  It follows that $\mathcal{B}(\mathfrak{c},G) = K[u]$ is the polynomial ring in one variable and is Gorenstein.
\end{proof}

In the proof of Theorem \ref{Gc}, it follows that $\mathcal{B}(k\mathfrak{c},G) = K[u^k]$ is also the polynomial ring in one variable, where $k$ is a positive integer.  On the other hand, however, even thought $\mathcal{B}(\mathfrak{c},G)$ is Gorenstein, one cannot expect that $\mathcal{B}(k\mathfrak{c},G)$ is Gorenstein.  Let $G$ be the path of length $2$ on $V(G) = \{x_1,x_2,x_3\}$ with the edges $\{x_1,x_2\}$ and $\{x_2,x_3\}$.  Then $\mathcal{B}((1,1,1),G)$ is Gorenstein, but $\mathcal{B}((3,3,3),G)$ is not Gorenstein.

\begin{Theorem}
\label{cG}
Given a vector $\mathfrak{c}\in \NN^n$, the toric ring  $\mathcal{B}(\mathfrak{c},K_n)$ is Gorenstein. In particular, for any vector $\mathfrak{c}\in \NN^n$,  there is a finite graph $G$ on $V(G)=\{x_1, \ldots, x_n\}$ for which $\mathcal{B}(\mathfrak{c},G)$ is Gorenstein.
\end{Theorem}

\begin{proof}
We only need to prove the first statement. Set $\delta:=\delta_{\mathfrak{c}}(I(K_n)$ and $u:=x_1^{c_1}\cdots x_n^{c_n}$.  One has $2\delta\leq \sum_{i=1}^nc_i$. If $2\delta=\sum_{i=1}^nc_i$, then $\mathcal{B}(\mathfrak{c},G)=K[u]$ is the polynomial ring in one variable and is Gorenstein.  If $2\delta=\sum_{i=1}^nc_i-1$, then $\mathcal{B}(\mathfrak{c},G)$ is generated by a subset of $\{u/x_{i}: i=1, \ldots, n\}$. Since the monomials $u/x_{1}, \ldots, u/x_{n}$ are algebraically independent, it follows that $\mathcal{B}(\mathfrak{c},G)$ is a polynomial ring in at most $n$ variables and is Gorenstein.

Now suppose that $2\delta \leq \sum_{i=1}^nc_i-2$.  One may assume that $c_1\geq c_2\geq\ldots \geq c_n$.  Let $v=e_1\cdots e_{\delta}=x_1^{a_1}\cdots x_n^{a_n}$ belong to $(I(K_n)^{\delta})_{\mathfrak{c}}$.  If there are integers $1\leq i < j\leq n$ with $a_i\leq c_i-1$ and $a_j\leq c_j-1$, then $v(x_ix_j)$ is a $\mathfrak{c}$-bounded monomial, contradicting the definition of $\delta$.  Thus, there is an integer $1\leq k\leq n$ for which $a_k\leq c_k-2$ and $a_{\ell}=c_{\ell}$ for each $\ell\neq k$.  If in the representation of $v$ as $v=e_1\cdots e_{\delta}$, there is an edge, say, $e_1=x_px_q$ which is not incident to $x_k$, then the monomial
$$vx_k^2=(x_px_k)(x_qx_k)e_2\cdots e_{\delta}$$ 
is a $\mathfrak{c}$-bounded monomial, contradicting the definition of $\delta$. Therefore, all the edges $e_1, \ldots, e_{\delta}$ are incident to $x_k$. Hence, $a_k=\delta$. If $k\geq 2$, then 
$$c_k\geq a_k+2=\delta+2\geq a_1+2=c_1+2$$
which contradicts our assumption $c_1\geq c_2\geq \ldots\geq c_n$.  Thus, $k=1$ and $v=x_1^{\delta}x_2^{c_2}\cdots x_n^{c_n}$.  It then follows that $\mathcal{B}(\mathfrak{c},G) = K[v]$ is the polynomial in one variable and is Gorenstein, as desired.
\end{proof}

\begin{Remark} \label{remcomp}
The proof of Theorem \ref{cG} shows that for any vector $\mathfrak{c}\in \NN^n$, the toric ring  $\mathcal{B}(\mathfrak{c},K_n)$ is isomorphic to a polynomial ring of dimension at most $n$. In Lemma \ref{dimcomp}, we show that the dimension of this toric ring is either one or $n$. 
\end{Remark}

\begin{Example}
\label{Ex}
Fix a positive integer $d$ and a vector $\mathfrak{a}=(a_1, \ldots, a_n)\in \NN^n$ with each $1 \leq a_i \leq d$ and $d \leq \sum_{i=1}^na_i$.  Recall from \cite{DH} that the {\em algebra of Veronese type} $A(d;\mathfrak{a})$ is the toric ring which is generated by those monomials $x_1^{q_1}\cdots x_n^{q_n}$ with each $q_i \leq a_i$ and with $\sum_{i=1}^nq_i = d$.  Let $G$ be the star graph with vertex set $V(G)=\{x_1, \ldots, x_{n+1}\}$ and $x_{n+1}$ its center. Assume that $(\mathfrak{a},d)$ denotes the vector of length $n+1$ which is defined as follows: for each $i=1, \ldots, n$, the $i$th component of $(\mathfrak{a},d)$ is $a_i$ and the last component of $(\mathfrak{a},d)$ is $d$. Since $\mathcal{B}((\mathfrak{a},d), G)$ is generated by those monomials of the form $ux_{n+1}^d$, where $u$ is a $\mathfrak{a}$-bounded monomial of degree $d$, it follows that $A(d;\mathfrak{a}) \cong \mathcal{B}((\mathfrak{a},d),G)$.   
\end{Example}

\section{Complete graphs}
Recall that Theorem \ref{cG} claims that the complete graph $K_n$ has the distinguished property that, for every $\mathfrak{c}\in \NN^n$, the toric ring $\mathcal{B}(\mathfrak{c},K_n)$ is Gorenstein.  One can ask if there is another class of finite graphs with this distinguished property. We answer this question in Theorem \ref{highlight}.

\begin{Lemma} \label{dimcomp}
Let $n\geq 3$ be an integer and $\mathfrak{c} = (c_1, \ldots, c_n)\in \NN^n$.  One has either $\dim\mathcal{B}(\mathfrak{c},K_n)=1$ or $\dim\mathcal{B}(\mathfrak{c},K_n)=n$.
\end{Lemma}

\begin{proof}
Set $\delta:=\delta_{\mathfrak{c}}(I(K_n)$. It follows from the proof of Theorem \ref{cG} that $\mathcal{B}(\mathfrak{c},K_n)$ is a polynomial ring and, if $\dim \mathcal{B}(\mathfrak{c},K_n)> 1$, then $2\delta=\sum_{i=1}^nc_i-1$.  Set $u:=\prod_{i=1}^nx_i^{c_i}$ and suppose that $\dim \mathcal{B}(\mathfrak{c},K_n)> 1$.  Let, say, $u/x_1 \in \mathcal{B}(\mathfrak{c},K_n)$.  We claim that $u/x_{k}\in \mathcal{B}(\mathfrak{c},K_n)$ for each $1 < k\leq n$.  This proves that $\dim\mathcal{B}(\mathfrak{c},K_n)=n$. Let $u/x_1=e_1\cdots e_{\delta}$, where $e_1, \ldots, e_{\delta}$ are edges of $K_n$.  If each $e_t$ is incident to $x_1$, then $\delta=c_1-1$ and it follows from $2\delta=\sum_{i=1}^nc_i-1$ that $c_1 = \sum_{j=2}^nc_j+1$.  Hence, $u/x_1$ is the only generator of $\mathcal{B}(\mathfrak{c},K_n)$ which implies that $\dim\mathcal{B}(\mathfrak{c},K_n)=1$. This is a contradiction, as we are assuming that $\dim\mathcal{B}(\mathfrak{c},K_n)> 1$. Thus, there is an edge $e_t$ with $1 \leq t \leq \delta$ which is not incident to $x_1$.  Let, say, $t=1$. Since $c_k \geq 1$, it follows that $x_k$  divides $u/x_1$. Thus, there is an integer $p$ with $1\leq p\leq \delta$ for which $e_p$ is incident to $x_k$. If $p=1$, then
$$u/x_k=(x_1e_1/x_k)e_2\cdots e_{\delta}\in\mathcal{B}(\mathfrak{c},K_n)$$
and we are done.  Suppose that $p\neq 1$, say, $p=2$. Then
$$u/x_k=(x_1e_1e_2/x_k)e_3\cdots e_{\delta}\in\mathcal{B}(\mathfrak{c},K_n),$$ as desired.
\end{proof}

\begin{Lemma}
\label{exvec}
Let $n\geq 3$ be an integer. Then there is a vector $\mathfrak{c} \in \NN^n$ for which $\mathcal{B}(\mathfrak{c},K_n)$ is isomorphic to the polynomial ring in $n$ variables over $K$.
\end{Lemma}

\begin{proof}
Let $n$ be odd and $\mathfrak{c}:=(1, \ldots, 1)$.  It then follows that $$\mathcal{B}(\mathfrak{c},K_n) = K[u/x_1, \ldots, u/x_n],$$
where $u=x_1\cdots x_n$.  Hence, $\mathcal{B}(\mathfrak{c},K_n)$ is the polynomial ring in $n$ variables. 

Let $n$ be even and $\mathfrak{c}=(2, \ldots, 2, 1)$.  It then follows that
$$\mathcal{B}(\mathfrak{c},K_n)=K[v/x_1, \ldots, v/x_n],$$
where $v=x_1^2\cdots x_{n-1}^2x_n$.  Hence, $\mathcal{B}(\mathfrak{c},K_n)$ is the polynomial ring in $n$ variables. 
\end{proof}

\begin{Lemma}
\label{nocomp}
Let $G$ be a finite graph on $V(G) = \{x_1, \ldots,x_n\}$ such that at least one connected component of $G$ is not a complete graph. Then there is a vector $\mathfrak{c}\in \NN^n$ for which the toric ring
$\mathcal{B}(\mathfrak{c},G)$ is not Gorenstein. 
\end{Lemma}

\begin{proof}
Let $x_p$ and $x_{p'}$ be non-adjacent vertices belonging to a connected component of $G$ which is not a complete graph.  Combining $x_p$ and $x_{p'}$ by a path in $G$, it follows that $G$ has non-adjacent vertices, say, $x_1, x_2$, which have a common neighbor, say $x_n$.  Furthermore, we assume that $|N_G(x_1)\cup N_G(x_2)|$ is smallest among all pairs of non-adjacent vertices with at least one common neighbor. Set $B:=N_G(x_1)\cup N_G(x_2)$. Let $A$ be the set of all vertices $x_i\in V(G)\setminus B$ for which $N_G(x_i)\subseteq B$. In particular, $x_1, x_2\in A$ and $x_n\notin A$. If two distinct vertices $x_i, x_j\in A$ are adjacent in $G$, then $x_i\in N_G(x_j)\subseteq B$, a contradiction. Thus, $A$ is an independent set of $G$.  Let $A=\{x_1, x_2, \ldots, x_m\}$, where  $2\leq m\leq n-1$.  For each $x_i\in A$, let $a_i$ denote the number of neighbors of $x_i$ in $B$ (which is equal to ${\rm deg}_G(x_i)$).  For each $x_i\in B$, let $b_i$ denote the number of neighbors of $x_i$ in $A$ (note that $b_i\geq 1$).  Moreover, for each $x_i\in V(G)\setminus (A\cup B)$, let $f_i$ denote the number of neighbors of $x_i$ in $V(G)\setminus (A\cup B)$.  If $x_i\in V(G)\setminus (A\cup B)$, then, since $N_G(x_i) \not\subset B$, there is a vertex $x_j \in N_G(x_i) \setminus B$.  Since $x_i \in N_G(x_j)$, one has $x_j \not\in A$.  Therefore, $x_j \in V(G)\setminus (A\cup B)$.  In other words, each $x_i\in V(G)\setminus (A\cup B)$ has at least one neighbor in $V(G)\setminus (A\cup B)$.  Thus, $f_i\geq 1$.  

Now, we introduce a vector $\mathfrak{c}=(c_1, \ldots, c_n)\in \NN^n$ defined by
\[
c_i=
\begin{cases}
  2a_i+2 & \text{if $x_i\in A$},\\
  2b_i & \text{if $x_i\in B$},\\
  f_i & \text{if $x_i\in V(G)\setminus (A\cup B)$}.
\end{cases}
\]
Set $\delta:=\delta_{\mathfrak{c}}(I(G))$.

\medskip

{\bf Claim 1.} $2\delta=(c_1+\cdots +c_n)-2|A|$.

\begin{proof}[Proof of Claim 1]
Let $\mathcal{E}_1$ denote the set of all edges $e$ of $G$ for which $e \cap A \neq \emptyset$ and $e \cap B \neq \emptyset$ and $\mathcal{E}_2$ the set of all edges $e$ of $G$ for which $e \subset V(G)\setminus (A\cup B)$.
Set$$u=\prod_{e\in \mathcal{E}_1}e^2\prod_{e\in\mathcal{E}_2}e,$$
where $\prod_{e\in\mathcal{E}_2}e=1$ if $\mathcal{E}_2=\emptyset$.  Then $u$ is a $\mathfrak{c}$-bounded monomial and$${\rm deg}(u)=2\sum_{x_i\in A}a_i+2\sum_{x_i\in B}b_i+\sum_{x_i\notin A\cap B}f_i=(c_1+\cdots +c_n)-2|A|.$$
It then follows that $2\delta\geq(c_1+\cdots +c_n)-2|A|$.  Let $v=e_1\cdots e_q$ be a $\mathfrak{c}$-bounded monomial, where $e_1, \ldots, e_q\in E(G)$.  Thus, for each $x_i\in V(G)\setminus (A\cup B)$ one has ${\rm deg}_{x_i}(v)\leq f_i$ and for each  $x_i\in B$ one has ${\rm deg}_{x_i}(v)\leq 2b_i$. Since $A$ is an independent set of $G$ with $N_G(A)=\cup_{x_i \in A} N_G(x_i) = B$, one has $$\sum_{x_i\in A}{\rm deg}_{x_i}(v)\leq \sum_{x_i\in B}{\rm deg}_{x_i}(v)\leq 2\sum_{x_i\in B}b_i=2\sum_{x_i\in A}a_i.$$Consequently,
\begin{eqnarray*}
{\rm deg}(v) &=& \sum_{x_i\in A}{\rm deg}_{x_i}(v)+\sum_{x_i\in B}{\rm deg}_{x_i}(v)+\sum_{x_i\notin A\cup B}{\rm deg}_{x_i}(v) \\
& \leq & 2\sum_{x_i\in A}a_i+2\sum_{x_i\in B}b_i+\sum_{x_i\notin A\cup B}f_i \\ &=& \sum_{i=1}^nc_i-2|A|.
\end{eqnarray*}
Hence, $2\delta=(c_1+\cdots +c_n)-2|A|$, as desired.
\end{proof}

Let $w \in (I(G)^{\delta})_\mathfrak{c}$ be a $\mathfrak{c}$-bounded monomial.  The above proof of Claim 1 shows that $w$ must be divisible by the monomial $$w' = \prod_{x_i\in B}x_i^{2b_i}\prod_{x_i\notin A\cup B}x_i^{f_i}$$ and $w=w'w''$, where $w''$ is a monomial on the variables $\{x_i : x_i\in A\}$ with 
$$\deg (w'') = \sum_{x_i\in A}{\rm deg}_{x_i}(w)=2\sum_{x_i\in A}a_i.$$
Moreover,$${\rm deg}_{x_i}(w'')\leq 2a_i+2$$ for each $x_i\in A$.

\medskip

{\bf Claim 2.} 
Let $u_0$ be a monomial on $\{x_i : x_i\in A\}$ with $\deg (u_0) = 2\sum_{x_i\in A}a_i$ and with ${\rm deg}_{x_i}(u_0)\leq 2a_i+2$ for each $x_i\in A$. Then $w'u_0\in (I(G)^{\delta})_\mathfrak{c}$.

\begin{proof}[Proof of Claim 2]
We first introduce the bipartite graph $H$ with the vertex set $V(H)=A'\sqcup B'$, where
\begin{eqnarray*}
&A':=\big\{x_{ij}: x_i \in A \ {\rm divides} \ u_0 \ {\rm and} \ 1\leq j\leq {\rm deg}_{x_i}(u_0)\big\},& \\
&B'=\{x_{ij}: x_i\in B \ {\rm and} \ 1\leq j\leq 2b_i\}.&
\end{eqnarray*}
The edges of $H$ are those $\{x_{st}, x_{k\ell}\}$, where $x_{st}\in A'$ and $x_{k\ell}\in B'$, for which $x_s\in A$ and $x_k\in B$ are adjacent in $G$.  Thus
$$|A'|={\rm deg}(u_0)=2\sum_{x_i\in A}a_i=2\sum_{x_i\in B}b_i=|B'|.$$
Our work is to show that $H$ has a perfect matching. By using Marriage Theorem \cite[Lemma 9.1.2]{HHgtm260}, it is enough to prove that for each nonempty subset $A''\subseteq A'$, one has $|N_H(A'')|\geq |A''|$. Let $\sigma (A'')$ be the set of those $x_i \in A$ for which there is $1 \leq j \leq {\rm deg}_{x_i}(u_0)$ with $x_{ij}\in A''$.  We consider the following two cases.

\smallskip

{\bf Case 1.} Suppose that $\sigma(A'')\subseteq \{x_1, x_2\}$. If $\sigma(A'')=\{x_1, x_2\}$, then $N_H(A'')=B'$ and the inequality $|N_H(A'')|\geq |A''|$ is trivial. Suppose that $|\sigma(A'')|=1$, say, $\sigma(A'')=\{x_1\}$. Then, since $x_1,x_2\in N_G(x_n)$, we deduce that
\begin{eqnarray*}
|A''|  &\leq& {\rm deg}_{x_1}(u_0)\leq 2a_1+2=2|N_G(x_1)|+2 \\
&=&2|N_G(x_1)\setminus\{x_n\}|+4 \leq 2\sum_{x_i\in N_G(x_1)\setminus\{x_n\}}b_i+2b_n \\
&=&2\sum_{x_i\in N_G(x_1)}b_i=|N_H(A'')|,
\end{eqnarray*}
as required.

\smallskip

{\bf Case 2.} Suppose that $\sigma(A'')\nsubseteq \{x_1, x_2\}$. If $\{x_1,x_2\}\subset \sigma(A'')$, then $N_H(A'')=B'$ and the inequality $|N_H(A'')|\geq |A''|$ is trivial. So, suppose that $\{x_1,x_2\}\nsubseteq \sigma(A'')$. Without loss of generality, we may assume that $x_2 \not\in \sigma(A'')$. If there are two distinct vertices $x_r, x_{r'}\in \sigma(A'')$ with $N_G(x_r)\cap N_G(x_{r'})\neq \emptyset$, then it follows from the minimlity of $|N_G(x_1)\cup N_G(x_2)|$ that $N_G(x_r)\cup N_G(x_{r'})=N_G(x_1)\cup N_G(x_2)=B$. Consequently, $N_H(A'')=B'$ and the inequality $|N_H(A'')|\geq |A''|$ is trivial.  Now, suppose that for any pair of distinct vertices $x_r, x_{r'}\in \sigma(A'')$, one has $N_G(x_r)\cap N_G(x_{r'})=\emptyset$. For each $x_r\in \sigma(A'')\setminus \{x_1\}$, one has $$a_r+1={\rm deg}_G(x_r)+1 \leq \sum_{x_i\in N_G(x_r)}{b_i},$$where the inequality follows from the fact that each vertex $x_i\in N_G(x_r)$ is adjacent to at least one of the vertices $x_1$ and $x_2$, and $x_1, x_2\neq x_r$. Moreover, since $x_n\in N_G(x_1)\cap N_G(x_2)$, one has $$a_1+1={\rm deg}_G(x_1)+1 \leq \sum_{x_i\in N_G(x_1)}{b_i}.$$
Hence, it follows from the above inequalities that
\begin{eqnarray*}
|A''| & \leq & \sum_{x_r\in \sigma(A'')}{\rm deg}_{x_r}(u_0)\leq \sum_{x_r\in \sigma(A'')}(2a_r+2) \\
&\leq& 2\sum_{x_r\in \sigma(A'')} \, \sum_{x_i\in N_G(x_r)}{b_i}\\ 
& = &2\sum_{x_i\in N_G(\sigma(A''))}{b_i}=|N_H(A'')|,
\end{eqnarray*}
where the first equality follows from the assumption that for any pair of distinct vertices $x_r, x_{r'}\in \sigma(A'')$, one has $N_G(x_r)\cap N_G(x_{r'})=\emptyset$. 

We conclude from Cases 1 and 2 above that $H$ has a perfect matching, say, $M$. For every edge $f=\{x_{st}, x_{k\ell}\}\in M$, set $\tau(f):=x_sx_k\in I(G)$. Recall from the proof of Claim 1 that $\mathcal{E}_2$  is the set of all edges $e$ of $G$ for which $e \subset V(G)\setminus (A\cup B)$. Then 
$$\prod_{f\in M}\tau(f) \prod_{e\in \mathcal{E}_2}e= u_0 \prod_{x_i \in B}x_i^{2b_i}\prod_{x_i\notin A\cup B}x_i^{f_i} \in (I(G)^\delta)_{\mathfrak{c}},$$ 
as desired.
\end{proof}

Let $\mathcal{M}$ denote the set of all monomials $v_0$ on $\{x_i : x_i\in A\}$ with  $\deg (v_0) = 2\sum_{x_i\in A}a_i$ and with ${\rm deg}_{x_i}(v_0)\leq 2a_i+2$ for each $x_i\in A$.  It follows from Claim 2 together with the argument after the proof of Claim 1 that 
$$\mathcal{B}(\mathfrak{c},G) = K[w'v_0: v_0\in \mathcal{M}].$$
Thus $\mathcal{B}(\mathfrak{c},G) \cong K[v_0: v_0\in \mathcal{M}]$. In other words, $\mathcal{B}(\mathfrak{c},G)$ is the algebra of Veronese type $A(d;\mathfrak{a})$, where $d=2\sum_{x_i\in A}a_i$ and $\mathfrak{a}=(2a_1+2, \ldots, 2a_m+2)\in \NN^m$.  Finally, \cite[Theorem 2.4]{DH} guarantees that $\mathcal{B}(\mathfrak{c},G)$ is not Gorenstein, as desired.
\end{proof}

Let, in general, $G$ be a finite graph on $V(G) = \{x_1, \ldots, x_n\}$ and suppose that $G$ is the disjoint union of $G_1$ on $V(G_1) = \{x_1, \ldots, x_m\}$ and $G_2$ on $V(G) = \{x_{m+1}, \ldots, x_n\}$.  Let $\mathfrak{c_1} = (c_1, \ldots, c_m), \mathfrak{c_2} = (c_{m+1}, \ldots, c_n)$ and $\mathfrak{c} = (c_1, \ldots, c_n)$.  It follows that
\[
\mathcal{B}(\mathfrak{c},G)=\mathcal{B}(\mathfrak{c_1},G_1) \# \mathcal{B}(\mathfrak{c_2},G_2),
\]
the Segre product of $\mathcal{B}(\mathfrak{c_1},G_2)$ and $\mathcal{B}(\mathfrak{c_2},G_2)$.  The next lemma follows from the criterion of Gorenstein rings of the Segre product \cite[Theorem 4.4.7]{GW}.

\begin{Lemma}
\label{GW}
Let $e > 1$ be an integer and let $S_i$ be the polynomial ring in $n_i$ variables over a field $K$ for each $1 \leq i \leq e$.  Then the Segre product $S_1 \# \cdots \# S_e$ is Gorenstein if and only if there is an integer $a>0$ for which $n_i \in \{1,a\}$ for every $i=1, \ldots, e$.
\end{Lemma}

We now classify all finite graphs $G$ on $V(G) = \{x_1, \ldots,x_n\}$ with the property that for each vector $\mathfrak{c}\in \NN^n$, the toric ring $\mathcal{B}(\mathfrak{c},G)$ is Gorenstein.

\begin{Theorem}
\label{highlight}
Let $G$ be a finite graph on $V(G) = \{x_1, \ldots,x_n\}$.  Then the toric ring $\mathcal{B}(\mathfrak{c},G)$ is Gorenstein for each vector $\mathfrak{c}\in \NN^n$ if and only if there is an integer $t\geq 3$ for which every connected component of $G$ is either $K_2$ or $K_t$.
\end{Theorem}

\begin{proof}
It follows from Lemma \ref{nocomp} that if $\mathcal{B}(\mathfrak{c},G)$ is Gorenstein for each vector $\mathfrak{c}\in \mathbb{N}^n$, then every connected component of $G$ is a complete graph. Let $G$ be the disjoint union of complete graphs $K_{n_1}, \ldots, K_{n_s}$ with each $n_i \geq 2$.  Lemma \ref{exvec} says that there is a vector $\mathfrak{c}\in \NN^n$ for which $\mathcal{B}(\mathfrak{c},G)$ is the Segre product of the polynomial rings
\[
S_{k_i} \# \cdots \# S_{k_s},
\]
where $k_i = 1$ if $n_i = 2$ and where $k_i = n_i$ if $n_i > 2$.  It then follows from Lemma \ref{GW} that if $\mathcal{B}(\mathfrak{c},G)$ is Gorenstein, then $n_i = n_j$ if $n_i > 2$ and $n_j > 2$.

Conversely, suppose that $n_i = n_j$ if $n_i > 2$ and $n_j > 2$.  If $n_i > 2$, then Lemma \ref{dimcomp} implies that for any $\mathfrak{c}\in \NN^{n_i}$, either $\mathcal{B}(\mathfrak{c},K_{n_i})$ is the polynomial ring in one variable or $\mathcal{B}(\mathfrak{c},K_{n_i})$ is the polynomial ring in $n_i$ variables.  Hence, by Lemma \ref{GW}, we deduce that for each vector $\mathfrak{c}\in \NN^n$, the toric ring $\mathcal{B}(\mathfrak{c},G)$ is Gorenstein, as desired.  
\end{proof}

\section{Classifications}
We now discuss Question \ref{question} for special classes of finite graphs. First, we consider the graphs obtained from a complete multipartite graph by deleting the edges of a (possibly empty) matching of it. 

Let $m \geq 2, n_1 \geq 1, \ldots, n_m \geq 1$ be integers and  
\[
V_i = \{x_{\sum_{j=1}^{i-1} n_j+1}, \ldots, x_{\sum_{j=1}^{i} n_j}\}, \, \, \, \, \, \, \, \, \, \, 1 \leq i \leq m.
\]
The finite graph $K_{n_1, \ldots, n_m}$ on $V(K_{n_1, \ldots, n_m}) = V_1 \sqcup \cdots \sqcup V_m$ with 
\[
E(K_{n_1, \ldots, n_m}) = \{ \{ x_k, x_\ell \} : x_k \in V_i, \, x_\ell \in V_{j}, \, 1 \leq i < j \leq m\}.
\]
is called the {\em complete multipartite graph} \cite[p.~394]{oh} of type $(n_1, \ldots, n_m)$.

By virtue of the criterion \cite[Theorem 2.4]{DH} of Gorenstein algebras of Veronese type, we can classify the vectors $\mathfrak{c}=(c_1, \ldots, c_{|V(G)|})\in\NN^{|V(G)|}$ for which $\mathcal{B}(\mathfrak{c},K_{n_1, \ldots, n_m}-M)$ is Gorenstein, where $M$ is a matching of $K_{n_1, \ldots, n_m}$.

\begin{Theorem} \label{completeAVT}
Let $m \geq 2, n_1 \geq 1, \ldots, n_m \geq 1$ be integers and $n = n_1 + \cdots + n_m$.  Let $K_{n_1, \ldots, n_m}$ be the complete multipartite graph of type $(n_1, \ldots, n_m)$.  Let $M$ be a matching of $K_{n_1, \ldots, n_m}$ such that the graph $G:= K_{n_1, \ldots, n_m} - M$ has no isolated vertex. Let $\mathfrak{c}=(c_1, \ldots, c_n)\in \NN^n$ and set
\[
\ell_i:=\sum_{x_h\in V_i}c_h, \, \, \, \, \, \, \, \, \, \, 1 \leq i \leq m.
\]
Furthermore, set 
\begin{eqnarray}
\label{ABCDEFG}
d_k:={\rm min}\{c_k, \sum_{x_{\ell}\in N_G(x_k)}c_{\ell}\}, \, \, \, \, \, \, \, \, \, \, 1 \leq k \leq n.
\end{eqnarray}
\begin{itemize}
\item[($\alpha$)] If there is $\{x_k,x_{k'}\}\in M$ for which 
\begin{eqnarray}
\label{ABCDEF}
c_k >\sum_{\substack{x_t\in N_G(x_k)\\ x_t\notin N_G(x_{k'})}} c_t, \, \, \, \, \, \, \, \, c_{k'}> \sum_{\substack{x_t\in N_G(x_{k'})\\ x_t\notin N_G(x_k)}}c_t, \, \, \, \, \, \, \, \, \, c_k+c_{k'}\geq \sum_{\substack{1\leq t\leq n\\t\neq k,k'}}c_t+2,
\end{eqnarray}
then $$\mathcal{B}(\mathfrak{c},G) \cong A(d;\mathfrak{f}),$$ where 
\[
d=\sum_{\substack{1\leq t\leq n\\t\neq k,k'}}c_t-\sum_{x_t\in N_G(x_k)\setminus N_G(x_{k'})}c_t-\sum_{x_t\in N_G(x_{k'})\setminus N_G(x_k)}c_t, 
\]
\[
\mathfrak{f}=(f_1, f_2), 
\]
with
\[
f_1:={\rm min}\Big\{c_k-\sum_{x_t\in N_G(x_k)\setminus N_G(x_{k'})}c_t, d\Big\},
\]
\[
f_2:={\rm min}\Big\{c_{k'}-\sum_{x_t\in N_G(x_{k'})\setminus N_G(x_k)}c_t, d\Big\}.
\]

\item[($\beta$)] Suppose that (\ref{ABCDEF}) fails to be satisfied for every edge of $\{x_k,x_{k'}\}\in M$.
\begin{itemize}
    \item[(i)] If$$\ell_i-2 < \sum_{\substack{1\leq \, j\leq m\\j\neq i}}\ell_j, \, \, \, \, \, \, \, \, \, \, 1 \leq i \leq m,$$then $\mathcal{B}(\mathfrak{c},G)$ is a polynomial ring in at most $n$ variables, and hence, it is Gorenstein.

    \item[(ii)] If there is $1 \leq i \leq m$ with $$\ell_i-2 \geq \sum_{\substack{1\leq \, j\leq m\\j\neq i}}\ell_j,$$
    then
   $$\mathcal{B}(\mathfrak{c},G) \cong A(d;\mathfrak{b}),$$ where 

\[
d=\sum_{\substack{1\leq \, j\leq m\\j\neq i}}\ell_j,
\]
\[
\mathfrak{b}=(d_{n_1+\cdots +n_{i-1}+1},\ldots, d_{n_1+\cdots +n_i})\in \NN^{n_{i}}. 
\]
\end{itemize}
\end{itemize}
\end{Theorem}

\begin{proof}
Assume that $V(G)=\{x_1, \ldots, x_n\}$. Set $\delta:=\delta_{\mathfrak{c}}(I(G))$.

\medskip

($\alpha$) Suppose that there is an edge $\{x_k,x_{k'}\}\in M$ for which \ref{ABCDEF} holds. Let $v=x_1^{a_1}\cdots x_n^{a_n}=e_1\cdots e_{\delta}$ be a minimal monomial generator of $(IG)^{\delta})_{\mathfrak{c}}$, where $e_1, \ldots, e_{\delta}$ are edges of $G$. Since $\{x_k,x_{k'}\}\in M$ in the representation of $v$ as $v=e_1\cdots e_{\delta}$, every edge $e_i$ is incident to at most one of $x_k$ and $x_{k'}$. Therefore, it follows from the third inequality of \ref{ABCDEF} that$$a_k+a_{k'}\leq \sum_{\substack{1\leq t\leq n\\t\neq k,k'}}{\rm deg}_{x_t}(v)\leq \sum_{\substack{1\leq t\leq n\\t\neq k,k'}}c_t\leq c_k+c_{k'}-2.$$ So, either $a_k\leq c_k-1$ or $a_{k'}\leq c_{k'}-1$. Without loss of generality, we may assume that $a_k\leq c_k-1$. First, assume that $a_{k'}=c_{k'}$. Then the above inequalities imply that $a_k\leq c_k-2$. Also, it follows from the second inequality of \ref{ABCDEF}  that in the representation of $v$ as $v=e_1\cdots e_{\delta}$, there is an edge, say, $e_1=\{x_{k'},x_p\}$, with $x_p\in N_G(x_k)\cap N_G(x_{k'})$. Replacing $v$ by$$(x_kv)/x_{k'}=(x_kx_p)e_2 \cdots e_{\delta},$$we may assume that $a_{k'}\leq c_{k'}-1$. Thus, in the sequel, we suppose that $a_k\leq c_k-1$ and $a_{k'}\leq c_{k'}-1$. Let $x_{\ell}\notin \{x_k, x_{k'}\}$ be an a vertex of $G$ with $a_{\ell}\leq c_{\ell}-1$. It follows from the structure of $G$ that either $\{x_k, x_{\ell}\}\in E(G)$ or $\{x_{k'},x_{\ell}\}\in E(G)$. In the first case, $(x_kx_{\ell})v\in (I(G)^{\delta+1})_{\mathfrak{c}}$ and in the second case $(x_{k'}x_{\ell})v\in (I(G)^{\delta+1})_{\mathfrak{c}}$. Both contradict the definition of $\delta$. So, $a_{\ell}=c_{\ell}$, for each vertex $x_{\ell}\notin\{x_k,x_{k'}\}$. In the representation of $v$ as $v=e_1\cdots e_{\delta}$, suppose that there is an edge, say, $e_1=\{x_r,x_s\}$ which is incident to neither $x_k$ nor $x_k$. Without loss of generality, we may assume that $\{x_k,x_r\}$ and $\{x_{k'},x_s\}$ are edges of $G$. It follows that 
$$(x_kx_{k'})v=(x_kx_r)(x_{k'}x_s)e_2\cdots e_{\delta}$$ is a $\mathfrak{c}$-bounded monomial in $I(G)^{\delta+1}$, a contradiction. Thus, for each $1\leq i\leq \delta$, the edge $e_i$ is incident to either $x_k$ or $x_{k'}$. Hence,$$v=x_k^{a_k}x_{k'}^{a_{k'}}\prod_{x_t\notin\{x_k,x_{k'}\}}x_t^{c_t}$$and $\delta=\sum_{x_t\notin\{x_k,x_{k'}\}}{c_t}$. Moreover, we conclude from the above argument that$$a_k \geq \sum_{\substack{x_t\in N_G(x_k)\\ x_t\notin N_G(x_{k'})}} c_t \, \, \, \, \, \, \, \, \, \, {\rm and} \, \, \, \, \, \, \, \, \, \, a_{k'}\geq \sum_{\substack{x_t\in N_G(x_{k'})\\ x_t\notin N_G(x_k)}}c_t.$$To simplify the notation, set$$m_k := \sum_{\substack{x_t\in N_G(x_k)\\ x_t\notin N_G(x_{k'})}} c_t \, \, \, \, \, \, \, \, \, \, {\rm and} \, \, \, \, \, \, \, \, \, \, m_{k'}:= \sum_{\substack{x_t\in N_G(x_{k'})\\ x_t\notin N_G(x_k)}}c_t.$$Consequently, $v$ can be written as$$v=v'x_k^{m_k}x_{k'}^{m_{k'}}\prod_{x_t\notin\{x_k,x_{k'}\}}x_t^{c_t},$$where using the notations introduced in the statement of the theorem, $v'$ is a $\mathfrak{f}$-bounded monomial of degree $d$ on variables $x_k$ and $x_{k'}$. Moreover, it is obvious that for any such a monomial $v'$, we have $v'x_k^{m_k}x_{k'}^{m_{k'}}\prod_{x_t\notin\{x_k,x_{k'}\}}x_t^{c_t}\in \mathcal{B}(\mathfrak{c},G)$. Thus, $\mathcal{B}(\mathfrak{c},G)$ is isomorphic to $A(d;\mathfrak{f})$, as desired.

\medskip

($\beta$) (i) As above, let $v=x_1^{a_1}\cdots x_n^{a_n}=e_1\cdots e_{\delta}$ be a minimal monomial generator of $(IG)^{\delta})_{\mathfrak{c}}$, where $e_1, \ldots, e_{\delta}$ are edges of $G$. Moreover, set $u=x_1^{c_1}\cdots x_n^{c_n}$. If $2\delta=\sum_{t=1}^{n}c_t$, then $\mathcal{B}(\mathfrak{c},G)=K[u]$ is the polynomial ring in one variable. If $2\delta=\sum_{t=1}^{n}c_t-1$, then $\mathcal{B}(\mathfrak{c},G)$ is generated by a subset of $\{u/x_{1}, \ldots, u/x_{n}\}$. Since $u/x_{1}, \ldots, u/x_n$ are algebraically independent, it follows that $\mathcal{B}(\mathfrak{c},G)$ is the polynomial ring in at most $n$ variables. So assume that $2\delta\leq \sum_{t=1}^{n}c_t-2$. If there are integers $1\leq i,j\leq n$ with $\{x_i, x_j\}\in E(G)$ for which $a_i\leq c_i-1$ and $a_j\leq c_j-1$, then $v(x_ix_j)$ is a $\mathfrak{c}$-bounded monomial, a contradiction. So, we have the following cases.

\smallskip

{\bf Case 1.} Suppose that there is an edge, say, $\{x_k, x_{k'}\}\in M$ with $a_k\leq c_k-1$ and $a_{k'}\leq c_{k'}-1$. Let $x_{\ell}\notin \{x_k,x_{k'}\}$ be an arbitrary vertex of $G$ and assume that $a_{\ell}\leq c_{\ell}-1$. It follows from the structure of $G$ that either $\{x_k, x_{\ell}\}\in E(G)$ or $\{x_{k'},x_{\ell}\}\in E(G)$. In the first case, $(x_kx_{\ell})v\in (I(G)^{\delta+1})_{\mathfrak{c}}$ and in the second case $(x_{k'}x_{\ell})v\in (I(G)^{\delta+1})_{\mathfrak{c}}$. Both contradict the definition of $\delta$. So, $a_{\ell}=c_{\ell}$, for each vertex $x_{\ell}\notin\{x_k,x_{k'}\}$. In the representation of $v$ as $v=e_1\cdots e_{\delta}$, suppose that there is an edge, say, $e_1=\{x_r,x_s\}$ which is incident to neither $x_k$ nor $x_{k'}$. We may assume that $\{x_k,x_r\}, \{x_{k'},x_s\}\in E(G)$. This yields that 
$$(x_kx_{k'})v=(x_kx_r)(x_{k'}x_s)e_2\cdots e_{\delta}$$ is a $\mathfrak{c}$-bounded monomial in $I(G)^{\delta+1}$, a contradiction. Thus, for each $1\leq i\leq \delta$, the edge $e_i$ is incident to either $x_k$ or $x_{k'}$.  Since $a_{\ell}=c_{\ell}$, for each vertex $x_{\ell}\notin \{x_k, x_{k'}\}$ and $a_k\leq c_k-1$ and $a_{k'}\leq c_{k'}-1$, our argument shows that the edge $\{x_k,x_{k'}\}\in M$ satisfies \ref{ABCDEF}, which is a contradiction.

\smallskip

{\bf Case 2.} Suppose that there is an integer $i$ with $1\leq i\leq m$ such that for each vertex $x_j$ with $a_j\leq c_j-1$, one has $x_j\in V_i$ (recall that $m$ denotes the number of parts in the vertex partition of $G$). In particular, $a_t=c_t$ for each  vertex $x_t\in V(G)\setminus V_i$. Suppose that in the representation of $v$ as $v=e_1\cdots e_{\delta}$, there is an edge, say, $e_1=\{x_r,x_s\}$ which is not incident to any vertex of $V_i$. Since  $2\delta\leq \sum_{t=1}^{n}c_t-2$, we conclude that either there are distinct vertices $x_{j_1},x_{j_2}\in V_i$ with $a_{j_1}\leq c_{j_1}-1$ and $a_{j_2}\leq c_{j_2}-1$, or there is a vertex $x_{j_0}\in V_i$ with $a_{j_0}\leq c_{j_0}-2$. In the first case, by the structure of $G$, one may assume that $\{x_r,x_{j_1}\}, \{x_s,x_{j_2}\} \in E(G)$. Thus,$$(x_{j_1}x_{j_2})v=(x_{j_1}x_r)(x_{j_2}x_s)e_2\cdots e_{\delta}\in (I(G)^{\delta})_{\mathfrak{c}},$$a contradiction. In the second case, if $\{x_r,x_{j_0}\}, \{x_s,x_{j_0}\}\in E(G)$, then by the same way as above, one derives a contradiction. So, without loss of generality, assume that $\{x_r,x_{j_0}\}\notin E(G)$. This means that $\{x_r,x_{j_0}\}\in M$. Moreover, we must have $\{x_s,x_{j_0}\}\in E(G)$. Then replacing $v$ by$$v'=(x_{j_0}v)/x_r=(x_{j_0}x_s)e_2\cdots e_{\delta}$$we are reduced to case 1 and the assertion follows from the argument in that case. So, we may assume that in the representation of $v$ as $v=e_1\cdots e_{\delta}$, every edge $e_k$ is incident to a vertex in $V_i$. This implies that$$\sum_{x_t\in V_i}c_t-2\geq \sum_{x_t\in V_i}a_t=\sum_{x_t\notin V_i}a_t=\sum_{x_t\notin V_i}c_t.$$In other words,$$\ell_i-2\geq \sum_{\substack{1\leq j\leq m\\ j\neq i}}\ell_j$$which contradicts our assumption.

\medskip

($\beta$) (ii) Let $d$ and the vector ${\mathfrak{b}}$ be as defined in the statement of the theorem. First assume that$$d>d_{n_1+\cdots +n_{i-1}+1},\ldots, d_{n_1+\cdots +n_i}.$$Then it follows from the assumption and definition of $d_{n_1+\cdots +n_{i-1}+1},\ldots, d_{n_1+\cdots +n_i}$ that there is a vertex $x_r\in V_i=\{x_{n_1+\cdots +n_{i-1}+1},\ldots, x_{n_1+\cdots +n_i}\}$ and a vertex $x_{r'}\in V(G)\setminus V_i$ such that $d_r= \sum_{x_{\ell}\in N_G(x_r)}c_{\ell}< c_r$ and $\{x_r,x_{r'}\}\in M$ and $d_j=c_j$ for each $j$ with $x_j\in V_i\setminus\{x_r\}$. It follows from $d>d_{n_1+\cdots +n_{i-1}+1},\ldots, d_{n_1+\cdots +n_i}$ that$$c_{r'}>\sum_{x_j\in V_i\setminus\{x_r\}}d_j=\sum_{x_j\in V_i\setminus\{x_r\}}c_j=\sum_{\substack{x_j\in N_G(x_{r'})\\ x_t\notin N_G(x_r)}}c_j.$$This implies that \ref{ABCDEF} holds for the edge $\{x_r,x_{r'}\}\in M$, a contradiction. Therefore, $$d\leq d_{n_1+\cdots +n_{i-1}+1},\ldots, d_{n_1+\cdots +n_i}.$$

\medskip

{\bf Claim.} ${\delta}=d$ and for each $\mathfrak{b}$-bounded monomial $w$ of degree $d$ on variables$$V_i=\{x_{n_1+\cdots +n_{i-1}+1},\ldots, x_{n_1+\cdots +n_i}\},$$the monomial $w\prod_{x_t\notin V_i}x_t^{c_t}$ belongs to $\mathcal{B}(\mathfrak{c},G)$.

\medskip
\noindent
{\em Proof of the claim.}
Let $w$ be a $\mathfrak{b}$-bounded monomial of degree $d$ on the variables $V_i=\{x_{n_1+\cdots +n_{i-1}+1},\ldots, x_{n_1+\cdots +n_i}\}$. We introduce the bipartite graph $H$ defined as follows.  
The vertex set is $V(H)=A\sqcup B$, where $$A=\{x_{st}: x_s\in V_i, 1\leq t\leq {\rm deg}_{x_s}(w)\}, \, \, B=\{x_{pq}: x_p\in V(G)\setminus V_i, 1\leq q\leq c_p\}.$$
Two vertices $x_{st}\in A$ and $x_{pq}\in B$ are adjacent in $H$ if  the vertices $x_s$ and $x_p$ are adjacent in $G$. Since ${\rm deg}(w)=d$, one has $|A|=|B|$. We prove that $H$ has a perfect matching. Using Marriage Theorem \cite[Lemma 9.1.2]{HHgtm260}, we show that for each nonempty subset $A'\subseteq A$, one has $|N_H(A')|\geq |A'|$. Set$$\sigma (A'):=\{x_k: \ {\rm there \ is \ an \ integer} \ r \ {\rm with} \ 1\leq r\leq {\rm deg}_{x_k}(w) \ {\rm such \ that} \ x_{kr}\in A'\}.$$If $|\sigma(A')|\geq 2$, then the structure of $G$ implies that $N_G(\sigma(A'))=V(G)\setminus V_i$. Therefore, $N_H(A')=B$, and the inequality $|N_H(A')|\geq |A'|$ is obvious in this case. So,   suppose that $\sigma(A')=\{x_r\}$ is a singleton. If the edges of $M$ are not incident to $x_r$, then $N_G(\sigma(A'))=V(G)\setminus V_i$ and again $N_H(A')=B$. Hence, suppose that there is a vertex $x_{r'}\in V(G)\setminus V_i$ such that $\{x_r,x_{r'}\}\in M$. Then$$N_G(\sigma(A'))=V(G)\setminus (V_i\cup\{x_{r'}\})$$and 
$$|N_H(A')|=\sum_{x_t\notin V_i}c_t-c_{r'}=\sum_{x_t\in N_{G}(x_r)}c_t\geq d_r\geq {\rm deg}_{x_r}(w)\geq |A'|,$$where the first inequality follows from the definition of $d_r$. 
Thus, $H$ has a perfect matching. Let $M'$ be a perfect matching of $H$. For every edge $f=\{x_{st}, x_{pq}\} \in M'$, set $\tau(f):=x_sx_p\in I(G)$. Then $\prod_{f\in M'}\tau(f)$ is equal to $w\prod_{x_t\notin V_i}x_t^{c_t}$. This shows that $\delta=d$ and  $$w\prod_{x_t\notin V_i}x_t^{c_t}\in \mathcal{B}(\mathfrak{c},G),$$ 
and the proof of the claim is complete.

\medskip

It follows from the claim that $\mathcal{B}(\mathfrak{c},G)$ is generated by all the monomials of the form $w\prod_{x_t\notin V_i}x_t^{c_t}$ where $w$ is a $\mathfrak{b}$-bounded monomial of degree $d$ on the variables$$V_i=\{x_{n_1+\cdots +n_{i-1}+1},\ldots, x_{n_1+\cdots +n_i}\},$$In other words, $\mathcal{B}(\mathfrak{c},G) \cong A(d;\mathfrak{b})$.
This completes the proof of the theorem.
\end{proof}

\begin{Example}
Let $M =\{\{x_1,x_4\}\}$ be a matching of $K_{3,2}$. Set $G=K_{3,2}-M$  and $\mathfrak{c}=(4,6,6,4,6)$. Then $\ell_1=16$ and $\ell_2=10$. Thus, the case ($\beta$)(ii) of Theorem \ref{completeAVT} occurs. Therefore, $\mathcal{B}(\mathfrak{c},G) \cong A(10;(4,6,6))$, which is not Gorenstein by \cite[Theorem 2.4]{DH}.
\end{Example}

Let $G$ be a graph on $n$ vertices. In the rest of this paper, we consider the vector $\mathfrak{c}=(1, 1, \ldots, 1)\in \NN^n$. In this case $\delta_{\mathfrak{c}}(I(G))$ is equal to the matching number of $G$. We first mention the following simple observation.

\begin{Proposition} \label{matchbig}
Let $G$ be a graph on $n$ vertices such that  ${\rm match}(G)\geq (n-1)/2$. Then for the vector $\mathfrak{c}=(1, 1, \ldots, 1)\in \NN^n$, the toric ring  $\mathcal{B}(\mathfrak{c},G)$ is Gorenstein. 
\end{Proposition}

\begin{proof}
Assume that $V(G)=\{x_1, \ldots, x_n\}$ and set $u:=x_1\cdots x_n$. If ${\rm match}(G)=n/2$, then 
$\mathcal{B}(\mathfrak{c},G)=K[u]$ is the polynomial ring in one variable and is Gorenstein. If ${\rm match}(G)=(n-1)/2$, then $\mathcal{B}(\mathfrak{c},G)$ generated by a subset of $\{u/x_1, u/x_2, \ldots, u/x_n\}$. Since $u/x_1, u/x_2, \ldots, u/x_n$ are algebraically independent, it follows that $\mathcal{B}(\mathfrak{c},G)$ is the polynomial ring in at most $n$ variables and is Gorenstein.
\end{proof}

Let $G$ be a graph on $n$ vertices and consider the vector $\mathfrak{c}=(1, 1, \ldots, 1)\in \NN^n$. In view of Proposition \ref{matchbig}, it is natural to ask for a characterization of graphs $G$ with ${\rm match}(G)=(n-2)/2$ such that $\mathcal{B}(\mathfrak{c},G)$ is Gorenstein. Answering this question looks difficult. However, we can answer it when $G$ is a tree (Theorem \ref{treegor}). We need the following lemmas. 

\begin{Lemma} \label{tree2}
Each vertex of a tree $T$ with $|V(T)| \geq 2$ is contained in the vertex set of a maximum matching of $T$.
\end{Lemma}

\begin{proof}
Let $M_1$ be a maximum matching of $T$ and fix a vertex $x\in V(T)$. If $x\in V(M_1)$, then we are done.  Suppose that $x\notin V(M_1)$. Since $x$ is not an isolated vertex of $T$, it has a neighbor $y$. If $y\notin V(M_1)$, then $M_1\cup\{\{x,y\}\}$ will be a matching of $T$, which is a contradiction, since $M_1$ is a maximum matching of $T$. Thus, $y\in V(M_1)$. Hence, there is $z \in V(T)$ with $e=\{y, z\}\in M_1$. Then $M:=(M_1\setminus\{e\})\cup\{\{x, y\}\}$ is a maximum matching of $T$ with $x \in V(M)$.
\end{proof}

\begin{Lemma} \label{treeone}
Let $G$ be a forest on $n$ vertices. Suppose that ${\rm match}(G)=(n-1)/2$ and that there are two vertices $y \neq z$ of $G$ such that, for every maximum matching $M$ of $G$, one has either $V(M)=V(G)\setminus\{y\}$ or $V(M)=V(G)\setminus\{z\}$. Then $y$ and $z$ are leaves of $G$ and there is $x\in V(G)$ with $\{x, y\} \in E(G)$ and $\{x, z\} \in E(G)$.  Furthermore, $G\setminus\{x, y, z\}$ has a perfect matching.
\end{Lemma}

\begin{proof}
Let $M_0$ be a maximum matching of $G$. Hence,  $V(G)\setminus V(M_0)$ is either $\{y\}$ or $\{z\}$. Let $V(G)\setminus V(M_0)=\{y\}$ and suppose that $N_G(y)=\{x_1, \ldots, x_k\}$. For each integer $p$ with $1\leq k\leq p$, we have $x_p\in V(M_0)$. Thus, there is an edge $e_p=\{x_p, x_{p'}\}\in M_0$. Then $M_p:=(M_0\setminus\{e_p\})\cup\{\{y,x_p\}\}$ is a maximum matching of $G$ and $V(G)\setminus V(M_p)=\{x_{p'}\}$. It follows from the hypothesis that $k=1$ and $x_{p'}=z$. Therefore, $y$ is a leaf of $G$. Moreover, $M_1=(M_0\setminus\{e_1\})\cup\{\{y,x_1\}\}$ and $V(G)\setminus V(M_1)=\{z\}$. Repeating the same process with $M_1$, we deduce that $z$ is also a leaf of $G$ and $y, z$ have the same unique neighbor $x:=x_1$. Since ${\rm match}(G)=(n-1)/2$, it follows that $G\setminus\{x, y, z\}$ has a perfect matching.
\end{proof}

A squarefree monomial ideal $I$ is called a {\em matroidal ideal} if there is a matroid $M$ on $\{x_1, \ldots, x_n\}$ such that $I$ is generated by all the monomials of the form $\prod_{x_i\in B}x_i$, where $B$ is a base of $M$. 

\begin{Lemma} \label{matr}
Let $I$ be a matroidal ideal of $S$ and set $u:=x_1\cdots x_n$. Assume that $\{v_1, \ldots, v_m\}$ is the minimal set of monomial generators of $I$. If $J$ is the monomial ideal generated by $\{u/v_1, u/v_2, \ldots, u/v_n\}$, then $J$ is a matroidal ideal.    
\end{Lemma}

\begin{proof}
Suppose $I$ is the matroidal ideal generated by the squarefree monomials corresponding to the bases of a matroid $M$. Then it is well-know that $J$ is the matroidal ideal associated to the so-called dual matroid of $M$.  
\end{proof}

{\bf Notation.}
Let $T$ be a tree with ${\rm match}(T)=(|V(T)|-1)/2$.  Then $\rho(T)$ stands for the number of vertices $x$ of $T$ for which $T-x$ has a perfect matching.

\begin{Example}
    Let $T$ be the path of length $4$ on the vertices $x_1,x_2,x_3,x_4,x_5$ with the edges $\{x_i,x_{i+1}\}$ for $i=1,2,3,4$.  Then ${\rm match}(T)=2=(5-1)/2$ and $T-x_i$ has a perfect matching if and only if $i= 1,3,5$.  Thus, $\rho(T) = 3$. 
\end{Example}

We are now ready to prove the last main result of this paper.

\begin{Theorem} \label{treegor}
Let $T$ be a tree on $n\geq 2$ vertices with ${\rm match}(T)=(n-2)/2$ and $\mathfrak{c}=(1, 1, \ldots, 1)\in \NN^n$.
\begin{itemize}
\item[(i)] If $T$ has a vertex which is adjacent to three leaves, then $\mathcal{B}(\mathfrak{c},T)$ is Gorenstein.

\item[(ii)] If $T$ has two distinct vertices such that each of these two vertices is adjacent to two leaves, then $\mathcal{B}(\mathfrak{c},T)$ is Gorenstein.

\item[(iii)] Suppose that there are eight vertices $x_1, \ldots, x_8$ of $T$ for which $$\{x_1, x_2\}, \{x_2, x_3\}, \{x_3, x_4\}, \{x_4, x_5\}, \{x_5, x_6\}, \{x_6, x_7\}, \{x_4, x_8\}$$ are edges of $T$, where $x_1, x_7, x_8$ are leaves of $T$ and ${\rm deg}_T(x_3)={\rm deg}_T(x_5)=2$.  Then $\mathcal{B}(\mathfrak{c},T)$ is Gorenstein.

\item[(iv)] Suppose that there are ten vertices $x_1, \ldots, x_{10}$ of $T$ for which 
\begin{eqnarray*}
 &\{x_1,x_2\}, \{x_2,x_3\}, \{x_3,x_4\}, \{x_4,x_5\}, \{x_5,x_6\}, &\\&\{x_6,x_7\}, \{x_4,x_8\}, \{x_8,x_9\}, \{x_9,x_{10}\}&   
\end{eqnarray*}
are edges of $T$, where $x_1, x_7, x_{10}$ are leaves of $T$ and ${\rm deg}_T(x_3)={\rm deg}_T(x_5)={\rm deg}_T(x_8)=2$. Then $\mathcal{B}(\mathfrak{c},T)$ is Gorenstein.

\item[(v)] Suppose that there are six vertices $x_1, \ldots, x_6$ of $T$ for which
$$\{x_1, x_3\}, \{x_2, x_3\}, \{x_3, x_4\}, \{x_4, x_5\}, \{x_5, x_6\}$$ 
are edges of $T$, where $x_1, x_2, x_6$ are leaves of $T$ and 
${\rm deg}_T(x_4)=2$. Then $\mathcal{B}(\mathfrak{c},T)$ is Gorenstein.

\item[(vi)] Let $T_1, \ldots, T_{\ell}$ $(\ell\geq 3)$ be trees for which ${\rm match}(T_i)=(|V(T_i)-1)/2$ for each $i\in\{1,2,3\}$, and each of the trees $T_4, \ldots, T_{\ell}$ has a perfect matching. Suppose that $\rho(T_1)=\rho(T_2)$.  Let $z_1$ (resp. $z_2$) be a vertex of $T_1$ (resp. $T_2$) for which $T_1-z_1$ (resp. $T_2-z_2$) has no perfect matching. Furthermore, suppose that $z_3$ is a vertex of $T_3$ for which $T_3-z_3$ has a perfect matching.  Let $z_4, \ldots, z_{\ell}$ be arbitrary vertices of $T_4, \ldots, T_{\ell}$, respectively.  Finally define $T$ to be the tree on the vertex set $V(T_1)\cup\cdots\cup V(T_{\ell})\cup\{x\}$, where $x$ is a new vertex, and with the edge set$$E(T)=\bigcup_{i=1}^{\ell}E(T_i)\cup\big\{\{x,z_i\}: 1\leq i\leq \ell\big\}.$$Then $\mathcal{B}(\mathfrak{c},T)$ is Gorenstein.

\item[(vii)]  Let $T_1, \ldots, T_{\ell}$ $(\ell\geq 3)$ be trees for which ${\rm match}(T_i)=(|V(T_i)-1)/2$ for each $i\in\{1,2,3\}$ and each of the trees $T_4, \ldots, T_{\ell}$ has a perfect matching. Suppose that $\rho(T_1)=\rho(T_2)+\rho(T_3)$.  Let $z_1$ be a vertex of $T_1$ for which  $T_1-z_1$ has no perfect matching. Furthermore, suppose that $z_2$ (resp. $z_3$) is a vertex of $T_2$ (resp. $T_3$) for which  $T_2-z_2$ (resp. $T_3-z_3$) has a perfect matching. Let $z_4, \ldots, z_{\ell}$ be arbitrary vertices of $T_4, \ldots, T_{\ell}$, respectively. Finally define $T$ to be the tree on the vertex set $V(T_1)\cup\cdots\cup V(T_{\ell})\cup\{x\}$, where $x$ is a new vertex, and with the edge set
$$E(T)=\bigcup_{i=1}^{\ell}E(T_i)\cup\big\{\{x,z_i\}: 1\leq i\leq \ell\big\}.$$Then $\mathcal{B}(\mathfrak{c},T)$ is Gorenstein.

\item[(viii)] If $T$ does not belong to the class of trees consisting of the trees described in (i)-(vii), then $\mathcal{B}(\mathfrak{c},T)$ is not Gorenstein.
\end{itemize}
\end{Theorem}

\begin{proof}
Set $u:=x_1\cdots x_n$. Since ${\rm match}(T)=(n-2)/2$, it follows that $\mathcal{B}(\mathfrak{c},T)$ is generated by monomials of the form $u/(x_ix_j)$ where $V(T)\setminus\{x_i,x_j\}$ is the vertex set of a maximum matching of $T$. Replacing $u/(x_ix_j)$ with $x_ix_j$, one sees that $\mathcal{B}(\mathfrak{c},T)$ is isomorphic to the toric ring generated by those squarefree monomials $x_ix_j$ for which $V(T)\setminus\{x_i, x_j\}$ is a maximum matching of $T$. In particular, $\mathcal{B}(\mathfrak{c},T)$ is isomorphic to the edge ring \cite{oh} of a finite graph $G$.

(i) Suppose that there is a vertex $x_i$ of $T$ which is adjacent to three leaves $x_p, x_q, x_r$. Since ${\rm match}(T)=(n-2)/2$, it follows that $\mathcal{B}(\mathfrak{c},T)$ is generated by the monomials $u/(x_px_q), u/(x_px_r)$ and $u/(x_qx_r)$, Thus, using the above argument, we deduce that $\mathcal{B}(\mathfrak{c},T)$ is the edge ring of the triangle, which is Gorenstein (\cite[Remark 2.8]{oh}).

(ii) Let $x_i \neq x_j$ be two vertices of $T$ and suppose that $x_i$ is adjacent to two leaves $x_{i_1}, x_{i_2}$ and that $x_j$ is adjacent to two leaves $x_{j_1},x_{j_2}$. Then $\mathcal{B}(\mathfrak{c},T)$ is generated by the monomials $u/(x_{i_1}x_{j_1}), u/(x_{i_1}x_{j_2}), u/(x_{i_2}x_{j_1})$ and $u/(x_{i_2}x_{j_2})$.  Hence, by the argument in the first paragraph of the proof, $\mathcal{B}(\mathfrak{c},T)$ is the edge ring of $K_{2,2}$, which is Gorenstein (\cite[Remark 2.8]{oh}).

(iii) Let $T_1 :=T_{\{x_1, \ldots, x_8\}}$ denote the induced subgraph of $T$ on $\{x_1, \ldots, x_8\}$. Also, set $T_2:=T - \{x_1, \ldots, x_8\}$. Let $M$ be a maximum matching of $T$. Thus, $|V(M)|=n-2$. If $x_1, x_7\in V(M)$, then, since $x_1$ and $x_7$ are leaves of $T$, we have $\{x_1,x_2\}, \{x_6,x_7\}\in M$. Hence, two of the vertices $x_3, x_5, x_8$ do not belong to $V(M)$. If $x_1\in V(M)$ and $x_7\notin V(M)$, then $\{x_1,x_2\}\in M$. Therefore, one of the vertices $x_3, x_8$ does not belong to $V(M)$. Similarly, if $x_1\notin V(M)$ and $x_7\in V(M)$, then one of the vertices $x_5, x_8$ does not belong to $V(M)$. In any case, one has $V(T)\setminus V(M)\subseteq \{x_1, \ldots,  x_8\}$. One can easily see that if there is $e \in M$ which is incident to a vertex of $T_1$ and a vertex in $T_2$, then $|V(M)|< n-2$, a contradiction. Since  $V(T)\setminus V(M)\subseteq \{x_1, \ldots, x_8\}$, we conclude that $M=M_1\cup M_2$, where $M_2$ is a perfect matching of $T_2$, and $M_1$ is a matching of $T_1$ with $|M_1| = 3$.  Thus, $\mathcal{B}(\mathfrak{c},T) \cong \mathcal{B}(\mathfrak{c'},T_1)$, where $\mathfrak{c'}=(1, \ldots, 1)\in \NN^8$. By the argument in the first paragraph of the proof, $\mathcal{B}(\mathfrak{c'},T_1)$ coincides with the toric ring of the complete multipartite graph $K_{2,2,1}$, which is Gorenstein (\cite[Remark 2.8]{oh}). 

(iv) Let $T_1 :=T_{\{x_1, \ldots, x_{10}\}}$ denote the induced subgraph of $T$ on $\{x_1, \ldots, x_{10}\}$. Also, set  $T_2:=T- \{x_1, \ldots, x_{10}\}$. Let $M$ be a maximum matching of $T$. Thus, $|V(M)|=n-2$. If $x_1, x_7\in V(M)$, then, since $x_1$ and $x_7$ are leaves of $T$, we have $\{x_1,x_2\}, \{x_6,x_7\}\in M$. Hence, two of the vertices $x_3, x_5, x_8, x_{10}$ do not belong to $V(M)$. If $x_1\in V(M)$ and $x_7\notin V(M)$, then $\{x_1,x_2\}\in M$. Therefore, one of the vertices $x_3, x_8, x_{10}$ does not belong to $V(M)$. Similarly, if $x_1\notin V(M)$ and $x_7\in V(M)$, then one of the vertices $x_5, x_8, x_{10}$ does not belong to $V(M)$. In any case, one has $V(T)\setminus V(M)\subseteq \{x_1, \ldots,  x_{10}\}$. It is easy to see that if there is $e \in M$ which is incident to a vertex of $T_1$ and a vertex in $T_2$, then $|V(M)|< n-2$, a contradiction. Since  $V(T)\setminus V(M)\subseteq \{x_1, \ldots, x_{10}\}$, we conclude that $M=M_1\cup M_2$ where $M_2$ is a perfect matching of $T_2$, and $M_1$ is a matching of $T_1$ with $|M_1| = 4$. Thus, $\mathcal{B}(\mathfrak{c},T) \cong \mathcal{B}(\mathfrak{c'},T_1)$, where $\mathfrak{c'}=(1, \ldots, 1)\in \NN^{10}$. By the argument in the first paragraph of the proof, $\mathcal{B}(\mathfrak{c'},T_1)$ coincides with the toric ring of the complete multipartite graph $K_{2,2,2}$, which is Gorenstein (\cite[Remark 2.8]{oh}).

(v) Let $T_1 :=T_{\{x_1, \ldots, x_6\}}$ denote the induced subgraph of $T$ on $\{x_1, \ldots, x_6\}$. Also, set $T_2:=T- \{x_1, \ldots, x_6\}$. Let $M$ be a maximum matching of $T$. Thus, $|V(M)|=n-2$. Since $x_1$ and $x_2$ are leaves of $T$ which have the same common neighbor $x_3$, it follows that $M$ cannot cover both $x_1$ and $x_2$.  If $x_1\in V(M)$ (resp. $x_2\in V(M)$), then $\{x_1,x_3\}\in M$ (resp. $\{x_2,x_3\}\in M$). Therefore, $M$ cannot cover both $x_4$ and $x_6$.  In any case, one has $V(T)\setminus V(M)\subseteq \{x_1, \ldots,  x_6\}$. One easily sees that if there is $e \in M$ which is adjacent to a vertex of $T_1$ and to a vertex in $T_2$, then $|V(M)|< n-2$, a contradiction. Since $V(T)\setminus V(M)\subseteq \{x_1, \ldots, x_6\}$, we conclude that $M=M_1\cup M_2$ where $M_2$ is a perfect matching of $T_2$, and $M_1$ is a matching of $T_1$ with $|M_1|=2$. Thus, $\mathcal{B}(\mathfrak{c},T) \cong \mathcal{B}(\mathfrak{c'},T_1)$, where $\mathfrak{c'}=(1, \ldots, 1)\in \NN^6$.  By the argument in the first paragraph of the proof, $\mathcal{B}(\mathfrak{c'},T_1)$ coincides with the toric ring of the complete multipartite graph $K_{2,1,1}$, which is Gorenstein (\cite[Remark 2.8]{oh}).

(vi) Every maximum matching $M$ of $T$ is of the form $$M=M_1\cup M_2\cup\cdots\cup M_{\ell}\cup\{x,z_3\},$$ where $M_1$ (resp. $M_2$) is a maximum matching of $T_1$ (resp. $T_2$), $M_3$ is a perfect matching of $T_3-z_3$ and $M_i$ is a perfect matching of $T_i$ for $4 \leq i \leq \ell$. Furthermore, $V(M)=V(T)\setminus\{x_{j_1}, x_{j_2}\}$, where $x_{j_1}$ (resp. $x_{j_2}$) can be any arbitrary vertex of $T_1$ (resp. $T_2$) for which $T_1-x_{j_1}$ (resp. $T_2-x_{j_2}$) has a perfect matching.  Again, by the argument in the first paragraph of the proof, $\mathcal{B}(\mathfrak{c},T)$ coincides with the toric ring of the complete bipartite graph $K_{\rho(T_1), \rho(T_2)}$. Since $\rho(T_1)=\rho(T_2)$, we conclude that $\mathcal{B}(\mathfrak{c},T)$ is Gorenstein (\cite[Remark 2.8]{oh}).

(vii) The proof of this part is omitted, as it is similar to the proof of (vi).

(viii) Suppose that $T$ does not belong to the class of those trees described in (i)-(vii) and that $\mathcal{B}(\mathfrak{c},T)$ is Gorenstein.

\medskip

{\bf Claim 1.} $T$ has a maximum matching $M$ for which $V(T)\setminus V(M)$ contains at least one non-leaf vertex.

\begin{proof}[Proof of Claim 1] 
Let $M_0$ be a maximum matching of $T$. If there is a non-leaf in $V(T)\setminus V(M_0)$, then we are done.  Suppose that $V(T)\setminus V(M_0)$ contains only two leaves $x_{k_1}$ and $x_{k_2}$. Let $x_{t_1} \in N_{T}(k_1)$ and $x_{t_2} \in N_{T}(x_{k_2})$.  It is possible that $x_{t_1}=x_{t_2}$. Since $x_{t_1}\in V(M_0)$, there is an edge  $e=\{x_{t_1}, x_{s_1}\}\in M_0$.  Note that $(M_0\setminus\{e\})\cup\{\{x_{t_1}, x_{k_1}\}\}$ is a matching of $T$ which does not cover  $x_{s_1}$. Hence, if $x_{s_1}$ is not a leaf of $T$, then we set $M:=M_0\setminus\{e\})\cup\{\{x_{t_1}, x_{k_1}\}\}$ and we are done. Suppose that $x_{s_1}$ is a leaf of $T$.  Similarly, one may also assume that there is a leaf $x_{s_2}$ of $T$ with $\{x_{t_2},x_{s_2}\}\in M_0$. If $x_{t_1}=x_{t_2}$, then it is adjacent to three leaves $x_{k_1}, x_{k_2}$ and $x_{s_1}$.  It follows that $T$ is a tree as described in (i), a contradiction. If $x_{t_1} \neq x_{t_2}$, then each $x_{t_i}$ is adjacent to two leaves $x_{k_i}$ and $x_{s_i}$.  Therefore, $T$ is a tree as described in (ii), a contradiction. 
\end{proof}

Let $M$ be a maximum matching of $T$ as described in Claim 1 and $x_k$ a non-leaf vertex in $V(T)\setminus V(M)$. Assume that $T_1, \ldots, T_c$ are the connected components of $T-x_{k}$. Then $c\geq 2$, as $x_k$ is not a leaf of $T$. For $1 \leq i \leq c$, let $H_i:= T_{V(T_i)\cup\{x_k\}}$ denote the induced subgraph of $T$ on $V(T_i)\cup\{x_k\}$. Since $x_k\notin V(M)$, one has ${\rm match}(T-x_k)=\frac{|V(T-x_k)|-1}{2}$. So, there is a connected component, say, $T_1$ of $T-x_k$ with ${\rm match}(T_1)=\frac{|V(T_1)|-1}{2}$ and ${\rm match}(T_i)=\frac{|V(T_i)|}{2}$ for $2 \leq i \leq c$.  In other words, each of the trees $T_2, \ldots, T_c$ has a perfect matching. Since ${\rm match}(T)=(n-2)/2$ and $x_k\notin V(M)$, it follows that ${\rm match}(H_1)=\frac{|V(H_1)|-2}{2}$. 

Consider the vector $\mathfrak{c}_1=(1, \ldots, 1)\in \NN^{|V(H_1)|}$. We know from \cite[Theorem 4.3]{HSF} (essentially, from \cite[Theorem 1 on page 246]{w}) that $(I(H_1)^{{\rm match}(H_1)})_{\mathfrak{c}_1}$ is a matroidal ideal. Thus, we conclude from Lemma \ref{matr}, 
\cite[Theorem 2.3]{kna} and the argument at the beginning of the proof that $\mathcal{B}(\mathfrak{c}_1,H_1)$ coincides with the toric edge ring of a complete multipartite graph $K_{r_1, \ldots, r_m}$ with $m\geq 2$.  Set $H:=K_{r_1, \ldots, r_m}$. Since $H_1$ has a maximum matching which does not cover $x_k$, we have $x_k\in V(H)$.

Recall the graph $G$ in the first paragraph of the proof. We now prove the following claims.

\medskip

{\bf Claim 2.} $H$ is a proper induced subgraph of $G$.

\begin{proof}[Proof of Claim 2] 
We first show that $H$ is a subgraph of $G$. Let $\{x_i,x_j\}\in E(H)$. This means that $H_1$ has a maximum matching $M_1$ with $V(H_1)\setminus V(M_1)=\{x_i, x_j\}$.  For each $2 \leq i \leq c$, consider a perfect matching $M_i$ of $T_i$. Then $M_1\cup M_2\cup \cdots \cup M_c$ is a maximum matching of $T$ which covers neither $x_i$ nor $x_j$. Hence, $\{x_i, x_j\}\in E(G)$. This implies that $H$ is a subgraph of $G$. We now show that $H$ is an induced subgraph of $G$. Let $x_{i'}$ and $x_{j'}$ be vertices of $H$ with $\{x_{i'}, x_{j'}\}\in E(G)$ and $M'$ a maximum matching of $T$ which covers neither $x_{i'}$ nor $x_{j'}$. Since $x_{i'}, x_{j'}\in V(H_1)$ and since $T_2, \ldots, T_c$ have perfect matchings, it follows that for an edge $e \in M'$, if $x_k \in e$, then $e$  is not incident to any vertex in $V(T_2) \cup \cdots \cup V(T_c)$ (where $x_k$ is the vertex introduced just after the proof of Claim 1). Thus, $M'\cap E(H_1)$ is a maximum matching of $H_1$ which covers neither $x_{i'}$ nor $x_{j'}$. Therefore, $\{x_{i'}, x_{j'}\}\in E(H)$ which proves that $H$ is an induced subgraph of $G$. Finally, we show that $V(H)$ is a proper subset of $V(G)$. Let $M_1'$ be a maximum matching of $T_1$ and $M_2'$ a maximum matching of $H_2$ which covers $x_k$ (the existence of $M_2'$ is guaranteed by Lemma \ref{tree2}). As above, for each $3 \leq i \leq c$, consider a perfect matching $M_i$ of $T_i$. Set $M'':=M_1'\cup M_2'\cup M_3\cup \cdots \cup M_c$, which is a maximum matching of $T$ and, in addition, there is a vertex of $T_2$ which is not covered by  $M''$. This means that a vertex of $T_2$ is contained in $V(G)\setminus V(H)$. Hence, $V(H)$ is a proper subset of $V(G)$, as desired.
\end{proof}

\medskip

{\bf Claim 3.} $|V(G)|-|V(H)|\geq c-1$ and $\{x_p, x_k\} \notin E(G)$ for each $x_p\in V(G)\setminus V(H)$.

\begin{proof}[Proof of Claim 3]
Let $M_1$ be a maximum matching of $T_1$ and $M_i$ a perfect matching of $T_i$ for $2 \leq i \leq c$.  For each $2 \leq i \leq c$, let $M_i'$ be a maximum matching of $H_i$ with $x_k \in V(M'_i)$ (the existence of $M_i'$ is guaranteed by Lemma \ref{tree2}). For each $i=2, \ldots, c$, there is a vertex $x_{p_i}\in V(T_i) \setminus V(M_i')$.  Then 
$$M_1\cup M_2 \cup\cdots\cup M_{i-1}\cup M_i'\cup M_{i+1}\cup\cdots\cup M_c$$ 
is a maximum matching of $T$ which does not cover $x_{p_i}$. Thus $|V(G)|-|V(H)|\geq c-1$, as desired.  

Now to prove the second part, let $x_p\in V(G)\setminus V(H)$ with $\{x_p,x_k\}\in E(G)$. This means that there is a maximum matching $M_0$ of $T$ with $V(T)\setminus V(M_0)=\{x_p,x_k\}$. Recall that $T_1$ is an odd component ({\em i.e.}, $|V(T_1)|$ is odd) of $T-x_k$ and each of $T_2, \ldots, T_c$ is an even components of $T-x_k$. Since $x_k\notin V(M_0)$, it follows that $M_0\cap (E(T_2)\cup\cdots\cup E(T_c))$ is a perfect matching of $T_2\cup\cdots\cup T_c$. In particular, $x_p\in V(T_1)$ and $M_0\cap E(H_1)$ is a maximum matching of $H_1$ which covers neither $x_p$ nor $x_k$. Consequently, $x_p\in V(H)$, a contradiction.
\end{proof}

{\bf Claim 4.} Let $x_t \in V(T_2)$ with $\{x_t, x_k\}\in E(T)$. If no leaf of $T_2$ is adjacent to $x_t$, then $|V(G)|-|V(H)|\geq c$.

\begin{proof}[Proof of Claim 4]
By the same argument as in the proof of Claim 3, it is enough to show that there are two vertices $x_{q_1}, x_{q_2}\in V(T_2)$, and two maximum matchings $M'$ and $M''$ of $H_2$ with $V(M')=V(H_2)\setminus\{x_{q_1}\}$ and $V(M'')=V(H_2)\setminus\{x_{q_2}\}$. The existence of $x_{q_1}$ (and $M'$) follows from the proof of Claim 3. By hypothesis, there is a vertex  $x_r \in N_{T_2}(x_{q_1})$  with $\{x_r, x_k\} \notin E(T)$. If $x_r\notin V(M')$, then $M'\cup \{\{x_r, x_{q_1}\}\}$ will be a matching of $H_2$ which is a contradiction, as $M'$ is a maximum matching of $H_2$. Thus, $x_r\in V(M')$, and so, there is an edge $e=\{x_r, x_{q_2}\}\in M'$. Note that $x_{q_2}\neq x_k$ and hence, $x_{q_2}\in V(T_2)$. Then $M''=(M'\setminus\{e\})\cup\{\{x_r,x_{q_1}\}\}$ is a maximum matching of $H_2$ with $V(M'')=V(H_2)\setminus\{x_{q_2}\}$.
\end{proof}

Recall that $H=K_{r_1, \ldots, r_m}$ with $m\geq 2$. Since $\mathcal{B}(\mathfrak{c},T)$ is Gorenstein, it follows from \cite[Remark 2.8]{oh} and Claim 2 that $m\leq 4$. We proceed our proof with dividing the situation into the following cases.

\medskip

{\bf Case 1.} Let $m=4$. Since by Claim 2, $H$ is a proper subgraph of $G$, using \cite[Remark 2.8]{oh}, we deduce that $\mathcal{B}(\mathfrak{c},T)$ is not Gorenstein, a contradiction.

\medskip

{\bf Case 2.} Let $m=3$. Suppose $V(H)=V_1\sqcup V_2\sqcup V_3$ with $|V_i|=r_i$ for $i=1, 2, 3$. By Claim 2, $H$ is an induced subgraph of $G$. Since $\mathcal{B}(\mathfrak{c},T)$ is Gorenstein, it follows from \cite[Remark 2.8]{oh} that $r_1, r_2, r_3\leq 2$. As we mentioned before Claim 2, $x_k$ is a vertex of $H$. Without loss of generality, we may assume that $x_k\in V_1$.

\smallskip

{\bf Subcase 2.1.} Let $r_1=2$. By Claim 3, there is $x_p\in V(G)\setminus V(H)$ for which $\{x_p,x_k\}\notin E(G)$. Then, in the partition of $V(G)$, the part containing $x_k$ has cardinality at least $3$. Thus, by using \cite[Remark 2.8]{oh}, $\mathcal{B}(\mathfrak{c},T)$ is not Gorenstein.

\smallskip

{\bf Subcase 2.2.} Let $r_1=r_2=1$ and $r_3=2$. Let $x_s$ denote the unique neighbor of $x_k$ in $H_1$. Moreover, assume that $V_2=\{y\}$ and $V_3=\{v,w\}$ with $y,v,w\in V(H_1)$. Every maximum matching $M_1$ of $H_1$ with $x_k\in V(M_1)$ contains the edge $\{x_s, x_k\}$ and its vertex set is either $V(H_1)\setminus\{y, v\}$ or $V(H_1)\setminus\{y, w\}$. In other words, there are only two possibilities for the vertex set of a maximum matching $M_1$ of $H_1$ with $x_k \in V(M_1)$.  Equivalently, there are two possibilities for the vertex set of a maximum matching of the graph $T_1-x_s$, as $N_{H_1}(x_k)=\{x_s\}$. Note that $T_1-x_s$ has no perfect matching, as otherwise $H_1$ has a perfect matching. Let $M_2$ be a maximum matching of $T_1-x_s$. Choose two vertices $x_{r_1}, x_{r_2}\in V(T_1-x_s)\setminus V(M_2)$. Suppose that $x_{r_1}$ and $x_{r_2}$ are not isolated vertices of $T_1-x_s$. Let $x_{q_1} \in N_{T_1-x_s}(x_{r_1})$ and $x_{q_2} \in N_{T_1-\{x_s\}}(x_{r_2})$.  It is possible that $x_{q_1}=x_{q_2}$.  For $i\in \{1,2\}$, there is an edge $e_i\in M_2$ which is incident to $x_{q_i}$.  Again it is possible that $e_1=e_2$.  Let $e_i=\{x_{q_i},x_{p_i}\}$.  Then $M_3:=(M_2\setminus \{e_1\})\cup\{\{x_{r_1},x_{q_1}\}\}$ and $M_4:=(M_2\setminus \{e_2\})\cup\{\{x_{r_2},x_{q_2}\}\}$ are maximum matchings of $T_1-x_s$. Hence, $T_1-x_s$ has at least three maximum matchings $M_2, M_3$ and $M_4$ with $V(M_i) \neq V(M_j)$ for $2\leq i,j\leq 4$ with $i \neq j$, a contradiction. This contradiction shows that at least one of the vertices $x_{r_1}$ and $x_{r_2}$ is an isolated vertex of $T_1-\{x_s\}$.  Suppose that $x_{r_1}$ is an isolated vertex of $T_1-x_s$. Since $T_1$ is connected, we conclude that $\{x_{r_1},x_s\} \in E(T_1)$. Our goal is to show that $T$ is a tree as described in (iii). Since there are two possibilities for the vertex sets of maximum matchings of $T_1-x_s$, it follows that there are two possibilities for the vertex sets of maximum matchings of $T_1-\{x_s, x_{r_1}\}$, as $x_{r_1}$ is an isolated vertex of $T_1-x_s$. Since $|V(T_1)|$ is odd, we deduce that $|V(T_1-\{x_s, x_{r_1}\})|$ is odd. If ${\rm match}(T_1-\{x_s, x_{r_1}\})\leq (|V(T_1-\{x_s, x_{r_1}\}|-3)/2$, then, since $x_{r_1}$ is a leaf of $T_1$ with $x_s \in N_{T_1}(x_{r_1})$, one has 
$${\rm match}(T_1)={\rm match}(T_1-\{x_s, x_{r_1}\})+1\leq \frac{(|V(T_1)|-3)}{2},$$ a contradiction.  Hence, $${\rm match}(T_1-\{x_s, x_{r_1}\})=(|V(T_1-\{x_s, x_{r_1}\}|-1)/2.$$ Therefore, by Lemma \ref{treeone}, $T_1-\{x_s, x_{r_1}\}$ has two leaves $x_1$ and $x_2$ with $x_3 \in N_{T_1-\{x_s, x_{r_1}\}}(x_1)$ and $x_3 \in N_{T_1-\{x_s, x_{r_1}\}}(x_2)$. If $\{x_s,x_1\}, \{x_s,x_2\}\notin E(T_1)$, then in $H_1$ there are two vertices $x_s$ and $x_3$, each of which is adjacent to two leaves. In fact, $x_s$ is adjacent to $x_k, x_{r_1}$ and $x_3$ is adjacent to $x_1, x_2$. Thus, the same argument as in the proof of (ii) says that $H$ is the complete bipartite graph  $K_{2,2}$, which is a contradiction, as $H=K_{1,1,2}$. Thus, $x_s$ is adjacent to at least one of $x_1$ and $x_2$. Furthermore, since $T$ has no cycle, $x_s$ is not adjacent to both of the vertices $x_1, x_2$. Suppose that $\{x_s,x_1\}\in E(T)$. Consequently, ${\rm deg}_T(x_1)=2$ and ${\rm deg}_T(x_2)=1$. By Claim 2, $H$ is a proper subgraph of $G$, and by Claim 3, for every $x_p\in V(G)\setminus V(H)$, one has $\{x_p,x_k\}\notin E(G)$. Since $\mathcal{B}(\mathfrak{c},T)$ is Gorenstein and $H=K_{1, 1, 2}$, it follows from \cite[Remark 2.8]{oh} that $G=K_{2,1,2}$. In particular, by Claim 3, one has $c=2$. Since $T$ is a tree, there is exactly one vertex $x_t\in V(T_2)$ with $\{x_k,x_t\}\in E(T)$. In particular, ${\rm deg}_T(x_k)=2$. By Claim 4, there is a leaf $x_{t'}$ of $T_2$ with $\{x_t,x_t'\}\in E(T)$. Thus, $T$ is a tree as described in (iii).

\smallskip

{\bf Subcase 2.3.}  Let $r_1=r_3=1$ and $r_2=2$. Then by a similar argument as in Subcase 2.2 (or by symmetry), a contradiction arises.

\smallskip

{\bf Subcase 2.4.} Let $r_1=1$ and $r_2=r_3=2$.  Let $x_s$ denote the unique neighbor of $x_k$ in $H_1$. Moreover, assume that $V_2=\{y, z\}$ and $V_3=\{v,w\}$.  Let $M_1$ be a maximum matching of $H_1$ with $V(M_1)=V(H_1)\setminus\{y,v\}$ and $M_2$ a maximum matching of $H_1$ with $V(M_2)=V(H_1)\setminus\{z,w\}$.  Since $x_k$ is a leaf of $H_1$ which is covered by $M_2$, we deduce that $\{x_k,x_s\}\in M_2$. Therefore, $z\neq x_s$ and $w\neq x_s$. Similarly, $\{x_k,x_s\}\in M_1$ and $y \neq x_s$ and $v\neq x_s$. Furthermore, as $y\in V(M_2)$, we deduce that $y$ is not an isolated vertex of $T_1-x_s$. Let $N_{T_1-x_s}(y)=\{x_{p_1}, \ldots x_{p_{\ell}}\}$. Assume that for some $i$ with $1\leq i\leq \ell$, we have $x_{p_i}\notin V(M_1)$. Then $M_1\cup\{\{y,x_{p_i}\}\}$ is a matching of $H_1$ which is a contradiction, as $M_1$ is a maximum matching of $H_1$. This contradiction shows that each vertex $x_{p_i}$ is covered by $M_1$. Hence, for each $i=1, \ldots, \ell$, there is an edge $e_i=\{x_{p_i},x_{q_i}\}\in M_1$. Since $x_{p_i}\neq x_s$ and $\{x_k,x_s\}\in M_1$, we have $x_{q_i}\in V(T_1-x_s)$. Then $M_{p_i}=(M_1\setminus\{e_i\})\cup\{\{y,x_{p_i}\}\}$ is a maximum matching of $H_1$ and $V(M_{p_i})=V(H_1)\setminus \{x_{q_i},v\}$. Since $H=K_{1,2,2}$ with $V_2=\{y, z\}$ and $V_3=\{v,w\}$, one has $\ell=1$ and $x_{q_1}=z$. In other words, $y$ is a leaf of $T_1-x_s$ and its unique neighbor $x_{p_1}$ is a neighbor of $z$. Similarly, $z$ is a leaf of $T_1-x_s$. To simplify the notation, set $x_1:=y, x_2:=x_{p_1}$ and $x_3:=z$. Thus, $x_1$ and $x_3$ are leaves of $T_1-x_s$ and $x_2$ is the unique neighbor of both of them. By the same argument, $v$ and $w$ are leaves of $T_1-x_s$ and they have the same unique neighbor. Set $x_5:=v, x_7:=w$ and let $x_6$ denote the unique (common) neighbor of $v$ and $w$. Since $T$ is not a tree as described in (i), we have $x_6\neq x_2$.  Set $x_4:=x_s$. Our goal is to show that $T$ is a tree as described in (iv). If $\{x_1,x_4\}, \{x_3, x_4\}\notin E(T_1)$, then $x_1$ and $x_3$ are leaves of $H_1$ and, since they have a common unique neighbor, it follows that every maximum matching of $H_1$ does not cover either $x_1$ or $x_3$. On the other hand, since $H=K_{1,2,2}$, $V_1=\{x_k\}$ and $V_3=\{v,w\}=\{x_5, x_7\}$, it follows that $H_1$ has a maximum matching which covers neither $x_k$ nor $x_5$, thus covers both $x_1, x_3$, which is a contradiction. Hence, we have either $\{x_1,x_4\}\in E(T_1)$ or $\{x_3,x_4\}\in E(T_1)$. Suppose that $\{x_3, x_4\}\in E(T_1)$.   Since $T$ has no cycle, $x_4$ is not adjacent to both of $x_1$ and $x_3$.  By the same argument, we may assume that $\{x_4,x_5\}\in E(T_1)$ and $\{x_4, x_7\}\notin E(T_1)$. In particular,  ${\rm deg}_T(x_1)={\rm deg}_T(x_7)=1$ and ${\rm deg}_T(x_3)={\rm deg}_T(x_5)=2$. Set $x_8:=x_k$. So, $\{x_4,x_8\}=\{x_s, x_k\}\in E(T)$. By Claim 2, $H$ is a proper induced subgraph of $G$ and, in addition, by Claim 3, for every $x_p\in V(G)\setminus V(H)$, one has $\{x_p, x_8\}=\{x_p, x_k\}\notin E(G)$. Since $\mathcal{B}(\mathfrak{c},T)$ is Gorenstein and $H=K_{1, 2, 2}$, it follows from \cite[Remark 2.8]{oh} that $G=K_{2,2,2}$. In particular, by Claim 3, one has $c=2$. Since $T$ is a tree, there is exactly one vertex, say, $x_9\in E(T_2)$ with $\{x_8,x_9\}=\{x_k,x_9\}\in E(T)$. In particular, ${\rm deg}_T(x_8)=2$. By Claim 4, there is a leaf, say, $x_{10}$ of $T_2$ with $\{x_9,x_{10}\}\in E(T)$. Thus, $T$ is a tree as described in (iv).

\smallskip

{\bf Subcase 2.5.} Let $r_1=r_2=r_3=1$. Then $V_1=\{x_k\}$. Suppose $V_2=\{y\}$ and $V_3=\{z\}$. Let $x_s$ denote the unique neighbor of $x_k$ in $H_1$. Also, let $M_1$ be a maximum matching of $H_1$ with $V(M_1)=V(H_1)\setminus\{y,z\}$. Since $M_1$ covers $x_k$ and $x_k$ is a leaf of $H_1$, we deduce that  $\{x_k,x_s\}\in M_1$. In particular, $y \neq x_s$ and $z\neq x_s$. Suppose that $N_{T_1-x_s}(y)=\{x_{p_1}, \ldots x_{p_{\ell}}\}$. Assume that for some $i$ with $1\leq i\leq \ell$, we have $x_{p_i}\notin V(M_1)$. Then $M_1\cup\{\{y,x_{p_i}\}\}$ is a matching of $H_1$ which is a contradiction, as $M_1$ is a maximum matching of $H_1$. This contradiction shows that each vertex $x_{p_i}$ is covered by $M_1$. Hence, there is  an edge $e_i=\{x_{p_i}x_{q_i}\}\in M_1$. Since $x_{p_i}\neq x_s$ and $\{x_k,x_s\}\in M_1$, we conclude that  $x_{q_i}\neq x_k$. Note that $M_{p_i}=(M_1\setminus\{e_i\})\cup\{\{y,x_{p_i}\}\}$ is a maximum matching of $H_1$ and $V(M_{p_i})=V(H_1)\setminus \{x_{q_i},z\}$. The existence of this maximum matching is a contradiction, as $H=K_{1,1,1}$ with $V_2=\{y\}$ and $V_3=\{z\}$.  Hence, $\ell=0$. In other words, $y$ is an isolated vertex of $T_1-x_s$. Since $T_1$ is a tree (in particular, connected), $y$ is a leaf of $T_1$ and $\{x_s,y\}\in E(T_1)$. By a similar argument, $z$ is a leaf of $T_1$ and $\{x_s,z\}\in E(T_1)$. To simplify the notation, set $x_1:=y, x_2:=z$ and $x_3:=x_s$. Therefore, ${\rm deg}_T(x_1)={\rm deg}_T(x_2)=1$. Our goal is to show that $T$ is a tree as described in (v). Set $x_4:=x_k$. Thus, $\{x_3,x_4\}=\{x_s,x_k\}\in E(T)$. By Claim 2, $H$ is a proper induced subgraph of $G$, and by Claim 3, for every $x_p\in V(G)\setminus V(H)$, one has $\{x_p,x_4\}=\{x_p, x_k\}\notin E(G)$. Since $\mathcal{B}(\mathfrak{c},T)$ is Gorenstein and $H=K_{1, 1, 1}$, it follows from \cite[Remark 2.8]{oh} that $G=K_{2,1,1}$. In particular, by Claim 3, one has $c=2$. Since $T$ is a tree, there is exactly one vertex, say, $x_5\in E(T_2)$ with $\{x_4,x_5\}=\{x_k,x_5\}\in E(T)$. In particular, ${\rm deg}_T(x_4)=2$. By Claim 4, there is a leaf, say, $x_6$ of $T_2$ with $\{x_5,x_6\}\in E(T)$. Hence, $T$ is a tree as described in (v).

\medskip

{\bf Case 3.} Let $m=2$.  Then $H=K_{r_1, r_2}$, where $r_1, r_2 >0$ are integers.  Since $\mathcal{B}(\mathfrak{c},T)$ is Gorenstein, it follows from \cite[Remark 2.8]{oh} and Claim 3 that $G$ is a complete bipartite graph. Suppose that $G=K_{s,t}$ and $V(G)=V_1\sqcup V_2$ with $|V_1|=s$ and $|V_2|=t$. We first show that $s,t\geq 2$.  Let $s=1$ and $V_1=\{y\}$. Then every maximum matching $M$ of $T$ does not cover $y$. This contradicts Lemma \ref{tree2}. Thus, $s,t\geq 2$. Hence, \cite[Remark 2.8]{oh} implies that $s=t$ and $G=K_{s,s}$. 

Since the number of vertices of $T$ is even and $T$ has no perfect matching, it follows from \cite[Exercise 5.3.3]{bm} that there is a vertex $x_1\in V(T)$ for which the number $k$ of odd connected components $T-x_1$ is at least $2$. However, $k$ cannot be even, as $|V(T-x_1)|$ is odd. Thus, $k\geq 3$. On the other hand, \cite[Exercise 5.3.4]{bm} implies that $k\leq 3$. Consequently, $k=3$. Let $L_1, L_2, \ldots, L_{\ell}$ with $\ell \geq 3$ denote the connected components of $T-x_1$, where $L_1, L_2, L_3$ are odd connected components of $T-x_1$ and $L_4, \ldots, L_{\ell}$ are even connected components of $T-x_1$. Since ${\rm match}(T)=(n-2)/2$, every maximum matching of $T$ contains an edge $e$ which is incident to $x_1$ as well as to a vertex in $V(L_1)\cup V(L_2)\cup V(L_3)$. In particular, each of $L_4, \ldots, L_{\ell}$ has a perfect matching and ${\rm match}(L_i)=(|V(L_i)|-1)/2$, for $i=1,2,3$. Furthermore, for every maximum matching $M$ of $T$, one has $V(M)=V(T)\setminus\{y, z\}$, where $y, z$ are vertices of distinct odd components of $T-x_1$. Let $M_1, M_2, \ldots, M_{\ell}$ be maximum matchings of $L_1, L_2, \ldots, L_{\ell}$, respectively. Thus, for each $i=1, 2, 3$, there is a vertex $y_i\in V(L_i)$ with $V(M_i)=V(T_i)\setminus \{y_i\}$. Let $z_1, z_2, z_3$ denote the unique neighbor of $x_1$ in $L_1, L_2, L_3$, respectively. Suppose that $L_1-z_1, L_2-z_2, L_3-z_3$ have perfect matchings, say, $M_1', M_2', M_3'$. Then for each pair of distinct integers $i, j\in \{1, 2, 3\}$, 
$$M_{ij}=M_i\cup M_j\cup M_h'\cup M_4\cup\cdots\cup M_{\ell}\cup\{\{x_1,z_h\}\}$$
is a maximum matching of $T$, where $h$ is the unique integer in $\{1, 2, 3\}\setminus\{i, j\}$. One has $V(M_{ij})=V(T)\setminus\{y_i, y_j\}$. Hence, $\{y_1,y_2\}, \{y_1,y_3\}, \{y_2,y_3\}\in E(G)$. This is a contradiction, as $G=K_{s,s}$ is a bipartite graph. This contradiction shows that at least one of the graphs $L_1-z_1, L_2-z_2, L_3-z_3$ has no perfect matching.  Without loss of generality, we may assume that $L_1-z_1$ has no perfect matching. If both $L_2-z_2$ and $L_3-z_3$ have no perfect matching, then each  maximum matching of $T$ does not cover at least one vertex in each of $L_1, L_2, L_3$, which contradicts ${\rm match}(T)=(n-2)/2$. Hence, either $L_2-z_2$ or $L_3-z_3$ has a perfect matching. We assume without loss of generality that $L_3-z_3$ has a perfect matching $M_3'$.

\smallskip

{\bf Subcase 3.1.} Suppose that $L_2-z_2$ has no perfect matching. For each maximum matching $M$ of $T$, one has $V(M)=V(T)\setminus \{v, w\}$, where $v \in V(L_1)$ for which $L_1-v$ has a perfect matching and $w \in V(L_2)$ for which $L_2-w$ has a perfect matching. Moreover, for such vertices $v$ and $w$, 
$$M_1''\cup M''_2\cup M_3'\cup M_4\cup\cdots\cup M_{\ell}\cup\{\{x_1, z_3\}\}$$
is a maximum matching of $T$ which covers neither $v$ nor $w$. Here, $M_1''$ is a perfect matching of $L_1-v$ and $M_2''$ is a perfect matching of $L_2-w$. Thus, $\{v, w\}\in E(G)$. Since $G=K_{s,s}$, one has $\rho(T_1)=s=\rho(T_2)$. Therefore, $T$ is a tree as described in (vi).

\smallskip

{\bf Subcase 3.2.} Suppose that $L_2-z_2$ has a perfect matching. For every maximum matching $M$ of $T$, one has $V(M)=V(T)\setminus \{v, w\}$, where $v \in V(L_1)$ for which $L_1-v$ has a perfect matching and $w \in V(L_j)$ with $j\in \{2, 3\}$ for which $L_j-w$ has a perfect matching. By a similar argument as in Subcase 3.1, for such vertices $v$ and $w$, one has $\{v, w\}\in E(G)$. Since $G=K_{s,s}$, one has $\rho(T_1)=s=\rho(T_2)+\rho(T_3)$. Thus, $T$ is a tree as described in (vii).
\end{proof}

\section*{Acknowledgments}
The second author is supported by a FAPA grant from Universidad de los Andes.

\section*{Statements and Declarations}
The authors have no Conflict of interest to declare that are relevant to the content of this article.

\section*{Data availability}
Data sharing does not apply to this article as no new data were
created or analyzed in this study.


\begin{thebibliography}{99}
\bibitem{bm} J.~Bondy and U.~Murty, {\em Graph Theory with Applications}, North-Holland, NewYork, 1976.
\bibitem{DH} E.~De Negri and T.~Hibi, Gorenstein algebras of Veronese type, {\it J. Algebra} {\bf 193} (1997),
    629--639.
\bibitem{GW} S.~Goto and K. ~Watanabe, On graded rings, I {\em J. Math. Soc. Japan} {\bf 30} (1978), 179--213.
\bibitem{HHgtm260}
J. Herzog and T. Hibi, {\em Monomial Ideals}, GTM 260, Springer, 2011.
\bibitem{HSF}
T.~Hibi and S.~A.~ Seyed Fakhari, Bounded powers of edge ideals: regularity and linear quotients, {\em Proc. Amer. Math. Soc.}, to appear, arXiv:2502.01768.
\bibitem{kna} K.~Khashyarmanesh, M.~Nasernejad and A.~Asloob Qureshi, On the matroidal path ideals, {\em J. Algebra Appl.}, {\bf 22} (2023), \#2350227.
\bibitem{oh} H.~Ohsugi and T.~Hibi, Compressed polytopes, initial ideals and complete multipartite graphs, {\em Illinois J. Math.} {\bf 44} (2000), 391--406.
\bibitem{w} D.~J.~A. Welsh, {\it Matroid Theory}, Academic Press, London, New York, 1976.
\end{thebibliography}
\end{document}